\documentclass[
 aip,
 jmp,
 amsmath,amssymb,
 reprint,
 onecolumn,
 11pt, a4paper,
 author-numerical,%
]{revtex4-1}

\usepackage{amsthm}
\usepackage{epsfig}
\usepackage{graphicx}

\newcommand{\be}{\begin{eqnarray}}
\newcommand{\ee}{\end{eqnarray}}
\newcommand{\nn}{\nonumber}
\newcommand{\bn}{\begin{eqnarray*}}       
\newcommand{\en}{\end{eqnarray*}}

\newcommand{\bea}{\begin{eqnarray*}}
\newcommand{\eea}{\end{eqnarray*}}
\newcommand{\ben}{\begin{eqnarray}}
\newcommand{\een}{\end{eqnarray}}
\newcommand{\beq}{\begin{equation}}
\newcommand{\eeq}{\end{equation}}
\newtheorem{definition}{Definition}
\newtheorem{remark}{Remark}
\newtheorem{theorem}{Theorem}
\newtheorem{lemma}{Lemma}
\newtheorem{proposition}{Proposition}
\newtheorem{corollary}{Corollary}

\newcommand{\C}{\ensuremath{\mathbb{C}}}
\newcommand{\Z}{\ensuremath{\mathbb{Z}}}
\newcommand{\R}{\ensuremath{\mathbb{R}}}
\newcommand{\N}{\ensuremath{\mathbb{N}}}

\newcommand{\la}{\langle}
\newcommand{\ra}{\rangle}

\newcommand{\Ne}{\mathcal{N}}

\newcommand{\bk}{\mathbf{k}}

\newcommand{\cl}[1]{\overline{#1}}
\newcommand{\eps}{\varepsilon}
\newcommand{\no}[1]{\left\Vert#1\right\Vert}
\newcommand{\til}[1]{\widetilde{#1}}
\newcommand{\cP}{\mathcal{P}}

\renewcommand{\Im}{\operatorname{Im}}
\renewcommand{\Re}{\operatorname{Re}}

\newcommand{\boldm}{\mathbf{m}}
\newcommand{\bx}{\mathbf{x}}
\newcommand{\bxi}{\boldsymbol{\xi}}

\newcommand{\bet}{\boldsymbol{\eta}}
\newcommand{\Ex}{\mathcal{E}}
\newcommand{\res}{\operatorname{res}}

\newcommand{\Lines}{\mathfrak{G}}
\newcommand{\Om}{\Omega}

\begin{document}

\title{Absence of bound states for waveguides
in 2D periodic structures}

\begin{abstract}
We study a Helmholtz-type spectral problem in a two-dimensional
medium consisting of a fully periodic background structure and a a
perturbation in form of a line defect. The defect is aligned along
one of the coordinate axes, periodic in that direction (with the
same periodicity as the background) and bounded in the other
direction. This setting models a so-called ``soft-wall'' waveguide
problem. We show that there are no bound states, i.e. the spectrum
of the operator under study contains no point spectrum.
\end{abstract}

\author{Vu Hoang}%
\email{duy.hoang@kit.edu, hoang@math.wisc.edu}
\affiliation{Institute for Analysis, Karlsruhe Institute for Technology (KIT)\\
Kaiserstrasse 89, 76133 Karlsruhe (Germany)/Department of Mathematics, University of Wisconsin, Madison, Wisconsin 53706 (USA)}%
\author{Maria Radosz}%
\email{radosz@math.uni-karlsruhe.de}
\affiliation{Institute for Analysis, Karlsruhe Institute for Technology (KIT),\\
Kaiserstrasse 89, 76133 Karlsruhe (Germany)}%

\date{1/21/2014}

\maketitle





\section{Introduction}

The study of defects in periodic materials such as semiconductors, photonic crystals or metamaterials
is a central issue in modern nanotechnology. Using point defects and line defects in an otherwise
perfect photonic crystal, for example, one can trap light or guide waves around sharp corners. It
is expected that these developments lead to novel optical devices to manipulate electromagnetic
waves, or even to an all-optical computer. For an introduction and a more thorough discussion of these matters,
we refer e.g. to the textbook \cite{Joann}, or to the review article \cite{Busch}.

Likewise, the study of the underlying partial differential equations
of mathematical physics, and in particular the associated
self-adjoint operators is of great mathematical importance. A
traditional problem in this field  is to investigate the spectral
properties of the operators, since the spectrum corresponds to the
physically admissible energies in the system.

There is one remarkable feature of the spectral theory of periodic
operators. While the absence of eigenvalues is easy
acceptable from an intuitive point of view, the rigorous
mathematical proofs in general require sophisticated techniques and
there are still many problems that have been inaccessible for a long
time.

In this paper, we solve one of these challenging problems.
Namely, the absence of eigenvalues of a two-dimensional partially periodic
Helmholtz-operator which models a soft-wall waveguide problem. In the following,
we describe the problem in greater detail.

Consider a spectral problem of Helmholtz-type on $\R^2$ of the form
\be
\label{eq1}
-\frac{1}{\varepsilon(\bx)} \Delta u = \lambda u.
\ee
The spatially variable dielectric function $\eps\in L^\infty(\R^2, \R)$ is
given by
\be\nn
\eps = \eps_0 + \eps_1
\ee
where $\eps_0$ is periodic
with respect to the lattice $\Z^2$ and $\eps_1$ is periodic in $x_2$
direction (with respect to $\Z$). Moreover,
\beq \label{suppeps}
\operatorname{supp} \eps_1 \subset (0,1)\times \R.
\eeq
$\eps$ and $\eps_0$ are bounded from below by a positive constant. The physical
interpretation is as follows: the spectral problem \eqref{eq1}
models the propagation of polarized electromagnetic waves in a
periodic medium, described by $\eps_0$, perturbed by a straight waveguide
(see figure \ref{figWave}).

The unperturbed spectral problem (i.e., $\eps$ replaced by $\eps_0$
in \eqref{eq1}) is periodic with respect to $\Z^2$ and its spectrum
has the well-known \emph{band gap} structure. Note that the
perturbed problem \eqref{eq1} is no longer periodic with respect to
the full lattice $\Z^2$. The perturbation may induce additional
spectrum. If this additional spectrum is present in a spectral gap of 
the unperturbed operator, then it should correspond to guided modes propagating in the direction of the waveguide \cite{Amm, Kuch3, Kuch4}. However, in
order to associate the additional spectrum with truly guided modes,
one has to prove that no eigenvalues of \eqref{eq1} are contained in
the band gaps. Such an eigenvalue corresponds to a localized mode
(bound state) on the whole space. Its existence should be highly
unlikely, according to physical intuition, but again, as mentioned
above, its nonexistence is in general very hard to prove
rigorously. Our main result is that the
spectrum of the self-adjoint operator $-\frac{1}{\eps(\bx)}\Delta$ (the
sense in which this operator is self-adjoint is made precise below)
contains no eigenvalues. From the applied point of view it is most interesting to exclude the 
existence of bound states in a gap, but our technique can easily handle the 
overall nonexistence of eigenvalues in the whole spectrum
of the perturbed operator.

Problems of a related nature have been a subject of intensive study for some decades.
We do not attempt to give a complete bibliography here and only mention a few key
contributions. For the Schr\"odinger operator with a potential periodic in all space directions,
the absolute continuity of the spectrum was proven in the celebrated paper by L.~Thomas
\cite{Thomas}. His results were extended to Schr\"odinger operators with magnetic potentials, by
M.Sh.~Birman and T.~Suslina in \cite{Bir} and by A.~Sobolev \cite{Sob1}. The periodic Maxwell
operator was treated by A.~Morame in \cite{Morame}. In the context of periodicity in all space
dimensions, the main obstacle is to prove the absence of eigenvalues; the absence of singular
continuous spectrum is true for a large class of periodic operators (see \cite{FiloSob}).  

An overview on results and open problems related to absolute continuity for periodic operators
is given in the papers \cite{Kuch2},\cite{Kuch6} and \cite{Sus}. The study of periodic waveguides goes back
to \cite{Der}. The problem of absolute continuity of the spectrum in periodic waveguides
with ``hard walls'' (i.e. where the guided modes are confined by e.g. Dirichlet boundary conditions)
has been considered in \cite{Sob}, \cite{Fried}, \cite{Sus2} and more recently, in
\cite{Kach}.

Although very relevant for modern developments in nanotechnology, for example photonic crystals
and quantum waveguides, there are only few mathematical publications dealing with ``soft wall'' or
``leaky'' waveguides, i.e. where the guided modes are allowed to
penetrate the surrounding medium with an exponential decay (see \cite{Kuch3}, \cite{Kuch4}).
Sufficient conditions for the existence of spectrum of \eqref{eq1} in 
spectral gaps of the periodic background
have been derived by H.~Ammari and F.~Santosa in \cite{Amm}, P.~Kuchment and B.~Ong in \cite{Kuch4}, \cite{Kuch5}
and also in \cite{BrownHoangPlumWood}.
The papers \cite{Filo, Filo2} by N.~Filonov and F.~Klopp treat a different type of ``soft-wall'' waveguide problem,
namely a periodic waveguide surrounded by a medium which is asymptotically homogeneous
in lateral direction. Another result \cite{ExnerFrank} by P.~Exner and R.~Frank
concerns the situation of leaky quantum waveguides. Here the surrounding medium has constant
material coefficients.

We would like to emphasize that in contrast to the situation in \cite{Filo, Filo2, ExnerFrank},
our waveguide is \emph{not} embedded into a background medium which has
constant material coefficients and neither a medium which
is asymptotically constant infinitely far from the defect in lateral direction.

To the authors' knowledge, the present paper contributes the first result on nonexistence of bound states in
periodic waveguides which are embedded into a background structure which is fully periodic with respect to
all space dimensions.
The problem is highly nontrivial, since the standard Thomas approach is not applicable
(see e.g.~the discussion in P.~Kuchment's review article \cite{Kuch5}). The main difficulty
comes precisely from the ``soft wall'' property of the problem (resulting
in a lack of compactness) and several new and rather sophisticated techniques are needed. We refer to section \ref{sec_plan} for an overview.

It would be interesting and desirable to prove the absolute continuity of the spectrum for the partially
periodic situation considered here. The result in \cite{FiloSob} 
on the absence of singular continous spectrum
unfortunately cannot be applied (see remark \ref{remOnSing}).

Finally, we would like to remark that our technique is applicable almost without changes to Schr\"odinger operators $-\Delta+V_0+V_1$ in two dimensions, where 
the potentials $V_0$ and $V_1$ have the properties of $\eps_0$ and $\eps_1$ respectively.

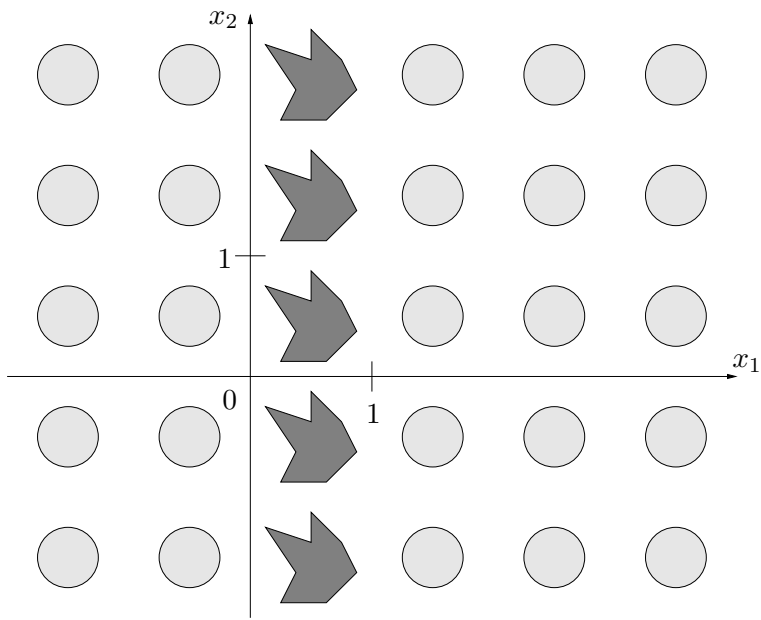
\begin{figure}[htbp]
\begin{center}
\caption{Illustration of the periodic waveguide. \label{figWave}}
\begin{picture}(0,0)
\includegraphics[scale=0.4]{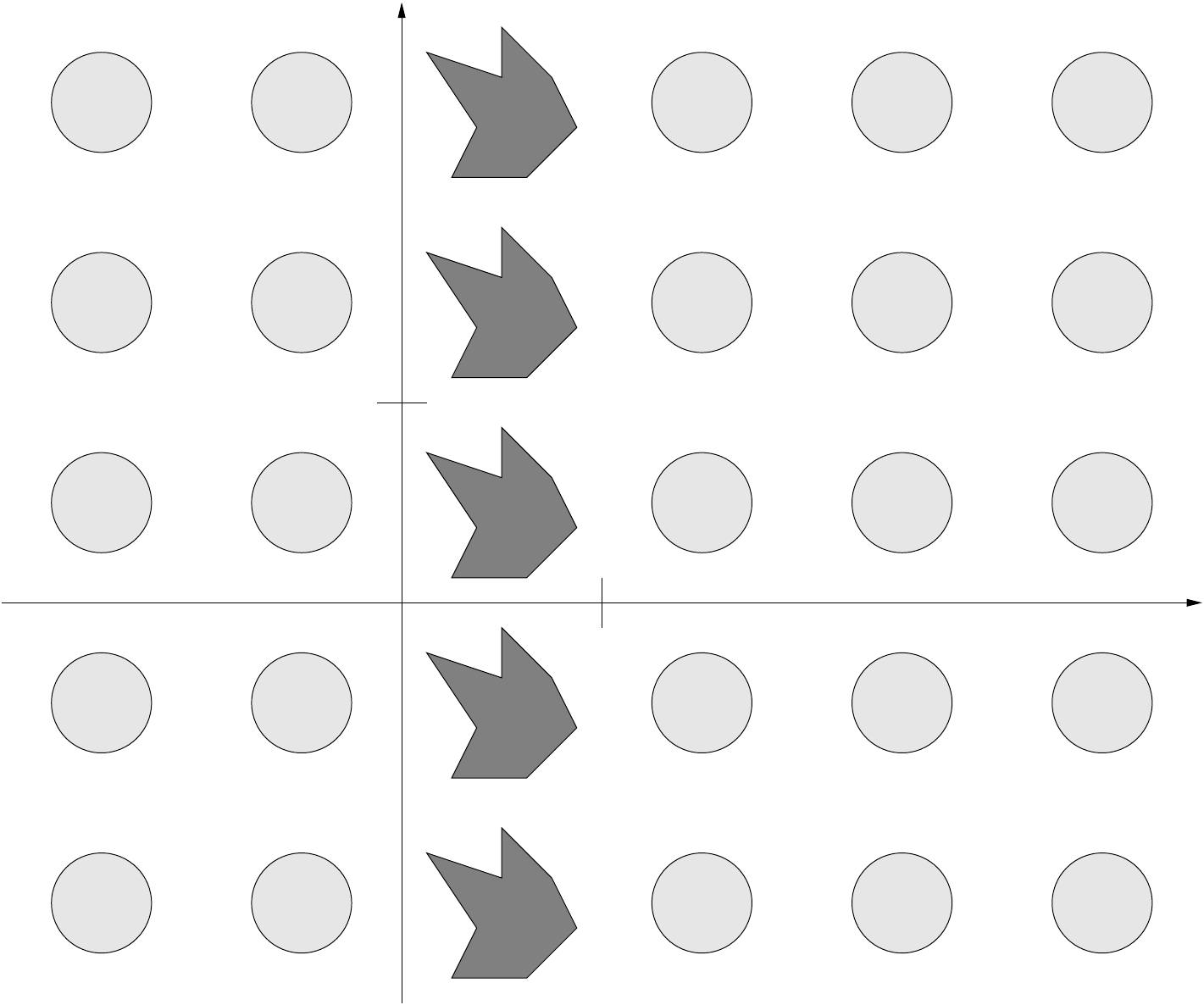}%
\end{picture}%
\setlength{\unitlength}{4144sp}%
\begingroup\makeatletter\ifx\SetFigFont\undefined%
\gdef\SetFigFont#1#2#3#4#5{%
  \reset@font\fontsize{#1}{#2pt}%
  \fontfamily{#3}\fontseries{#4}\fontshape{#5}%
  \selectfont}%
\fi\endgroup%
\begin{picture}(4000,4000)(0,0)
\put(4300,1500){\text{$x_1$}}
\put(1200,3550){\text{$x_2$}}
\put(1280,1250){\text{$0$}}
\put(2130,1170){\text{$1$}}
\put(1250,2090){\text{$1$}}
\end{picture}%
\end{center}
\end{figure}

\section{Notation and preliminaries.}

\subsection{Floquet-Bloch transformations.}
We introduce some notation that will be used below. Let
$S := \R\times (0,1)$ be the strip and $\Omega = (0,1)^2$ the unit
cell. Bold letters will indicate vectors in $\R^2$ or $\Z^2$, for example
$\bk=(k_1,k_2)$, $\boldm=(m_1,m_2)$, $\bet=(\eta_1,\eta_2)$. All
operator norms will be denoted by $\no{\cdot}$, since it will be
clear from the context on which spaces the operator acts in each
case.

Let $\eps_0, \eps_1 \in L^\infty(\R^2, \R)$. $\eps_0$ is assumed to be
periodic with respect to $\Z^2$, whereas for $\eps_1$ we assume
$$
\eps_1(x_1, x_2+m) = \eps_1(x_1, x_2) \quad (m\in \Z)
$$
and $\operatorname{supp} \eps_1\subset (0,1)\times \R.$ Both
$\eps_0$ and $\eps:=\eps_0+\eps_1$ shall be bounded from below by positive
constants, and moreover we assume that there exists a nonempty open set $\mathcal{M}$
with
$$
\operatorname{essinf}_\mathcal{M} |\eps_1| > 0.
$$
The assumption that $\operatorname{supp} \eps_1$ is contained in $(0,1)\times \R$ is made only for
convenience; more generally, we only need to assume that $\operatorname{supp} \eps_1$ is bounded
in $x_1$-direction.

$H_\text{per}^k(S)$ denotes the Sobolev space of functions periodic
in $x_2$-direction, and $H_\text{per}^k(\Omega)$ denotes the Sobolev space
of periodic functions on the unit cell.

The operator
$-\frac{1}{\eps(\bx)} \Delta$ is defined via the quadratic form
$$
b[u, v] = \int_{\R^2} \nabla u \overline{\nabla v}~d\bx, \quad D(b) = H^1(\R^2)\subset L^2(\R^2, \eps)
$$
in the \emph{weighted space} $L^2(\R^2, \eps)$ with inner product
$\la u, v\ra_\eps = \int_{\R^2} \eps(\bx) u \overline{v}~d\bx$. By well known standard arguments
(see \cite{Kato}) there exists
a self-adjoint realization of $-\frac{1}{\eps(\bx)} \Delta$ in $L^2(\R^2,\eps)$.
A simple regularity argument shows that the domain of this self-adjoint operator
is $H^2(\R^2)$. In the following we use
``$-\frac{1}{\eps(\bx)} \Delta$'' to denote the self-adjoint operator defined in the above way.

Note, that $-\frac{1}{\eps(\bx)} \Delta$ is self-adjoint in $L^2(\R^2, \eps)$ and that
$-\frac{1}{\eps_0(\bx)}\Delta$ is self-adjoint in $L^2(\R^2, \eps_0)$.

We will need ``shifted" Laplacian operators on $S$ and on $\Omega$.
For $k_2\in \C$, $-\Delta_{k_2}$ will denote the operator
\be\nn -\Delta_{k_2} := - (\nabla + i (0,k_2))\cdot (\nabla + i  (0,k_2))
\ee acting on functions in $H^2_{\text{per}}(S)$. For $\bk\in
\C^2$, $-\Delta_\bk$ denotes \be\nn -\Delta_\bk := - (\nabla + i
\bk)\cdot (\nabla + i \bk) \ee with $H^2_{\text{per}}(\Omega)$ as
domain.

\begin{remark}
An alternative approach, which works equally well after obvious alterations,
is to consider operators of the type $-\frac{1}{\sqrt{\eps(\bx)}}\Delta \frac{1}{\sqrt{\eps(\bx)}}$,
which would then be self-adjoint in the usual, unweighted $L^2$-inner product.
\end{remark}

The Floquet-Bloch transform in $x_2$ direction (see \cite{Kuch1}, \cite{Kuch2} for a general treatment of
Floquet-Bloch transforms)
\be\nn
(V_{x_2} f)(x_1, x_2, k_2) := \frac{1}{\sqrt{2\pi}} \sum_{n\in \Z} e^{i k_2 (n-x_2)} f(x_1, x_2-n)
\ee
maps $L^2(\R^2, \eps)$ isometrically onto $L^2((-\pi,\pi), L^2(S, \eps))$. We regard
$V_{x_2} f$ as a function mapping $k_2\in (-\pi, \pi)$ to $V_{x_2} f(\cdot,\cdot, k_2)\in L^2(S, \eps)$.
Note here that $\eps$ is periodic in $x_2$-direction and that
$V_{x_2} f(\cdot,\cdot, k_2+2\pi m)=V_{x_2} f(\cdot,\cdot, k_2)$ for $m\in\Z$. The inverse of $V_{x_2}$ is given by
$$
(V_{x_2}^{-1})g(\bx)=\frac{1}{\sqrt{2\pi}}\int_{-\pi}^\pi e^{ik_2x_2}g(\bx,k_2)~dk_2
$$
where $g(\bx,k_2)$ is extended periodically in $x_2$-direction.

It is well-known that $V_{x_2}$ can be
used to reduce the spectral problem \eqref{eq1} to the strip $S$; namely, $-\frac{1}{\eps(\bx)}\Delta$
can be expressed as a direct integral of operators
\beq\label{directS}
-\frac{1}{\eps(\bx)}\Delta = \int_{[-\pi,\pi)}^{\bigoplus} -\frac{1}{\eps(\bx)}\Delta_{k_2} ~dk_2
\eeq
and as a consequence, the spectrum of the self-adjoint operator $-\frac{1}{\eps(\bx)}\Delta$ is
decomposed into the union of the spectra of problems on the strip:
$$
\sigma\left(-\frac{1}{\eps(\bx)}\Delta\right) = \cl{\bigcup_{k_2\in [-\pi, \pi]}
\sigma\left(-\frac{1}{\eps(\bx)}\Delta_{k_2}\right)}.
$$

Using the full periodicity, on the other hand, the spectrum of the periodic operator $-\frac{1}{\eps_0} \Delta$
is decomposed according to
$$
\sigma\left(-\frac{1}{\eps_0(\bx)}\Delta\right) = \bigcup_{\bk\in [-\pi, \pi]^2}
\sigma\left(-\frac{1}{\eps_0(\bx)}\Delta_{\bk}\right).
$$

We will also use the following
Floquet-Bloch transform in $x_1$-direction on $S$:
\beq\label{Vx1}
(V_{x_1}f)(x_1, x_2, k_1) = \frac{1}{\sqrt{2\pi}} \sum_{n\in \Z} e^{i k_1 (n-x_1)} f(x_1-n, x_2).
\eeq
$V_{x_1} : L^2(S, \eps_0) \to L^2((-\pi, \pi), L^2(\Omega, \eps_0))$ is an isometry, too. Its inverse is
$$
(V_{x_1}^{-1}g)(\bx)=\frac{1}{\sqrt{2\pi}}\int_{-\pi}^\pi e^{ik_1x_1}g(\bx,k_1)~dk_1
$$
where $g(\bx,k_2)$ is extended periodically in $x_1$-direction.
As a useful fact, note that the Floquet transform of a function $f$, which is zero outside $\Omega$
is just
\beq\label{Vcomp}
(V_{x_1} f)(\bx, k_1) = (2\pi)^{-1/2} e^{-i k_1 x_1} f(\bx)
\eeq
and so $k_1 \mapsto (V_{x_1} f)(\cdot, k_1)$ is analytic in $k_1\in \C$ with values in $L^2(\Omega)$.

For $\lambda\in \R$ the inverse operator
of $(-\Delta_{\bk}-\lambda \eps_0)$ will frequently appear below,
and we write
\beq\label{defT}\nn
T(\bk) = T(k_1, k_2) := \frac{1}{2\pi}(-\Delta_{\bk}-\lambda \eps_0)^{-1}
\eeq
whenever $(-\Delta_{\bk}-\lambda \eps_0)^{-1}$ exists. This can be true or not, depending on
$\bk\in\C^2$. Furthermore, we often consider
$T(k_1, k_2)$ as a function of $k_1$ for fixed $k_2$. If for fixed $k_2\in \C$, $T(k_1, k_2)$
exists for some $k_1\in \C$, then $k_1 \mapsto T(k_1, k_2)$ is meromorphic
(see section \ref{sec::mero} for further explanation).

As usual, it will be important to diagonalize $-\Delta_{\bk}$ using
Fourier series on $\Omega$. Any $u\in H^2_{\text{per}}(\Omega)$ can
be expanded into Fourier modes $\{e^{i \boldm \cdot
\mathbf{x}}\}_{\boldm\in 2\pi \Z^2}$. On the level of Fourier
coefficients, the action of $-\Delta_{\bk}$ on $u$ is given by
multiplication with the symbol
$$s(\boldm,\bk)=(\boldm + \bk)^2.$$

\subsection{Meromorphic operator-valued functions.}\label{sec::mero}
As a convenience for the reader, we recall briefly
the concept of a meromorphic family of operators. Consider functions
$z \mapsto R(z)$ mapping complex numbers into the space of bounded linear operators
on a Hilbert space. $R(z)$ is called meromorphic if $R(z)$ is defined and analytic
on an open set $\mathcal{D}\subset \C$ except for a discrete set of points in
$\mathcal{D}$. If $z_0$ is one of these exceptional points, then we assume that
$R(z)$ has a Laurent series expansion
\beq\nn
R(z) = \sum_{n=-N}^\infty T_n (z-z_0)^n
\eeq
($0\leq N < \infty$) converging in the uniform operator topology in some punctured neighborhood
of $z_0$. Any point $z_0$ such that $R(z)$ has a Laurent series expansion with $N\geq 1$
is called a \emph{pole} of $R(z)$.

The reader can find information on the properties of
analytic and meromorphic operator-valued functions, for instance, in the book \cite{Hille} of Hille
and Phillips.

\begin{proposition}\label{propmero}
Let $k_2\in \C$ be fixed and suppose that
$T(k_1, k_2)=\frac{1}{2\pi}(-\Delta_{\bk}-\lambda \eps_0)^{-1}$
exists for one $k_1\in \C$. Then $T(k_1, k_2)$ exists for
all $k_1$ except for a discrete set of points in
the complex plane, these exceptional points being the poles of
$T(\cdot, k_2)$. Moreover,
$k_1 \mapsto T(k_1, k_2)$ is an operator-valued meromorphic function.
\end{proposition}
The proof is given in the appendix.

In what follows, we have to study the set of poles of $T(\cdot, k_2)$ as $k_2$ varies
in the complex plane. In a slightly different context, this question has been extensively
studied by S. Steinberg in \cite{Stei}. In general, the individual poles of $T(\cdot, k_2)$ move
continuously as $k_2$ varies, which follows from an adaption of results by S. Steinberg
(see \cite{Stei}) to our special case. Following the language of Steinberg, we say that
\emph{the poles} of $T(\cdot, k_2)$ are continuous functions of $k_2$.

Moreover, as long as the poles do not collide, they are given by analytic functions of
$k_2$. For certain values of $k_2$, two or several poles of $T(\cdot, k_2)$
may collide and in the neighborhood of these points, the poles are given by algebroidal
functions of $k_2$ (Puiseux series). We will prove all these facts and render them
precise without using Steinberg's results, to the extent they are needed here,
in the appendix.

\begin{remark}
The poles behave analogously to the eigenvalues of a matrix eigenvalue problem
of the form $M(\kappa) u = \lambda(\kappa) u$ with a matrix depending analytically
on $\kappa$ (see \cite{Kato}), which might be more familiar with most of the readers.
\end{remark}

\section{Main result and general plan of the paper.}\label{sec_plan}

\subsection{The main result}
\begin{theorem}\label{mainth}
The self-adjoint operator $-\frac{1}{\eps(\bx)}\Delta$ has no point spectrum.
\end{theorem}

\begin{remark}\label{remOnSing}
As mentioned before, it would be desirable to prove the absolute continuity of the spectrum of $-\frac{1}{\eps(\bx)}\Delta$
by excluding the singular continuous part of the spectrum. This would require, for instance, a result analogous to the one obtained by N. Filonov and A. Sobolev in \cite{FiloSob}; their result cannot be applied directly, since their proof requires the resolvents of the operators in the direct integral
decomposition \eqref{directS} to be compact. This is not the case here, since the operators act on functions defined on the unbounded
strip $S$. We leave the problem open for the moment.  
\end{remark}

We now sketch our approach to the proof of theorem \ref{mainth}.
\begin{enumerate}
\item[\textbf{1.}] The first step is to use the usual Floquet-Bloch reduction in $x_2$-direction.
Applied to \eqref{eq1}, this yields a problem on the strip $S$.
Thus, the existence of a nontrivial solution of \eqref{eq1} implies that
\beq\label{eq2}
(-\Delta_{k_2} - \lambda \eps(\bx)) \widetilde{v} = 0, \quad \widetilde{v}\in H^2_\text{per}(S)
\eeq
has a nontrivial solution for $k_2$ from a set $\cP$ with positive measure in $[-\pi,\pi]$. One cannot apply
Thomas' idea (extension to complex $k_2$) directly to \eqref{eq2}, since $S$ is unbounded
and hence the spectrum of the strip problem is not discrete.

However, using Floquet-Bloch transform with respect to the $x_1$-direction and using \eqref{suppeps} we
can derive a Fredholm problem on the unit cell $\Omega$:
\be\label{eq3}
v - \lambda A(k_2) \varepsilon_1 v = 0
\ee
with some compact operator $A(k_2) : L^2(\Omega)\to L^2(\Omega)$ to be introduced below
in \eqref{eqdefA}. $A(k_2)$ is defined for $k_2$ in a neighborhood of the real axis.

\item[\textbf{2.}] In the second step, we construct an analytic continuation of the operator family $A(k_2)$
to values $k_2$ with large imaginary part. Basically, the idea consists in using the integral representation
\beq\label{eqA} A(k_2)r=\int_{[-\pi,\pi]+i\delta_0} e^{ik_1x_1}T(k_1,k_2)[e^{-ik_1\cdot}r]~dk_1\eeq
and deforming the integral to obtain a new representation involving an integral over a line lying sufficiently
far away from the real axis (in the $k_1$-plane) plus a sum over the residues of the
meromorphic operator-valued function $T(\cdot,k_2)$. Here, the restriction to two space dimensions comes into play. Corresponding 
to the two space dimensions we have two complex quasimomenta $k_1$ and $k_2$. Note that $A(k_2)$ is given by a 
line integral in the complex $k_1$-plane. Thus the usual techniques from complex analysis in one 
variable are available, e.g.~the residue theorem. We will see that the integral and the residues from
the new representation of $A(k_2)$ are analytic in $k_2$, if $k_2$ is from a suitable region in the complex plane.
Since in the following we want to
let $\Im k_2\to \infty$, it is therefore
crucial to understand the movement of the poles of $T(\cdot,k_2)$
as $k_2$ varies. In general, however, the poles have algebraic singularities as functions
of $k_2$. In order to overcome this difficulty, we construct an analytic continuation only
in the neighborhood of a certain path in the complex plane, carefully avoiding the
algebraic branching points.
\item[\textbf{3.}] In the third and technically most difficult step, we study the
behavior of the analytically continued operator-valued family $A(k_2)$ for values $k_2$
with large imaginary part. By carefully estimating the symbol
of the shifted cell Laplacian $-\Delta_{\bk}$, we will be able to localize the poles of
$T(\cdot, k_2)$ for $\Im k_2$ large. The essential technical estimates are contained
in theorem \ref{prop1}, and here the two-dimensionality of the problem is required, too.
Finally, these estimates allow us to conclude by a Neumann
series argument that \eqref{eq3} has only the trivial solution.
\end{enumerate}

\subsection{Plan of the paper}
The remainder of the paper is structured as follows: in section \ref{sec_reform}, we
give the analytic Fredholm equation involving the operator $A(k_2)$. Then, in section
\ref{sec_cont} we describe in detail the analytic continuation process. The study
of the continued operator family for large imaginary values of $k_2$ occupies section
\ref{sec_asymp}. Finally, section \ref{sec_proofmain} contains the proof of the
main result. Since the argument is fairly difficult, we help the reader to keep
track of the main line of argumentation by transferring the technical details into
the appendix.

\section{Reformulation of the problem}\label{sec_reform}
Let $\lambda\in \R$ be fixed and $u\in H^2(\R^2)$ be a fixed nontrivial solution of \eqref{eq1}.
For the rest of the paper we will work with this fixed $u$, until finally in section \ref{sec_proofmain}
we will be led to a contradiction, thus proving the main result.
Since $u\in H^2(\R^2)$ solves
\be\label{eq6}
-\frac{1}{\eps(\bx)}\Delta u-\lambda u=0,
\ee
we deduce using the isometry property
of $V_{x_2}$ that $\til{v}(\cdot,k_2)=(V_{x_2}u)(\cdot,k_2)$
solves
\be\label{eq71}
(-\Delta_{k_2}-\lambda\eps(\bx))\til{v}(\cdot,k_2)=0\quad~\text{on}~~S
\ee
for almost all $k_2\in[-\pi,\pi]$. Since $u\neq 0$, the set
$\widetilde \cP$
of all real $k_2$ such that \eqref{eq71} has a nontrivial solution has positive one-dimensional
Lebesgue measure in $[-\pi, \pi]$.
Notice that the complex conjugate $\cl{\til{v}(\cdot, k_2)}$ solves
$(-\Delta_{-k_2} -  \lambda \eps)\cl{\til{v}(\cdot, k_2)} = 0$. Hence $-\widetilde \cP = \widetilde\cP$ and thus
\be\label{eqP}
\cP:=\widetilde \cP\cap(0,\pi]
\ee has positive measure in $[-\pi,\pi]$. Clearly there exists a
$$0< \theta < \pi$$ such that
$\cP \cap [\theta, \pi]$ has positive one-dimensional Lebesgue measure. We fix $\theta$ for the rest of the paper.

\subsection{The poles of $T(\cdot,k_2)$}
Starting from now, we fix a number
$\delta > 0$ such that
$$0 < \delta < \min\{\frac{\pi}{4}, \pi - \theta\}.$$
In the following, we define two domains $Z$ and $Z_0$ in the complex $k_2$-plane,
whose meaning will become clear later.
$Z$ is the following domain (see Figure \ref{figZZ0}):
\bea\nn Z &:=& \{z \in \C : \Re z \in (\pi-\delta,
\pi+\delta), \Im z \in \R \}\, \cup\, \{ z\in \C : \Re z \in (\theta, \pi+\delta), |\Im z| < \delta \}.
\eea

First we must make sure that $T(\cdot, k_2)$ is meromorphic for all $k_2\in Z$.
\begin{theorem}\label{freeofpoles}
There exists a number $\tau_1\in 2\pi \N$ such that for all $k_2\in Z$, $T(k_1, k_2)$ exists for all
$
k_1 \in [-\pi, \pi] \pm i\tau_1.
$
\end{theorem}
\begin{proof}
From theorem \ref{prop1} (in particular estimate \eqref{hammer1}) in
the appendix we get \be\nn \|-\Delta_{(\xi_1 + i \tau_1,k_2)}^{-1}\|
\leq \left(\min_{m_2\in 2\pi\Z}|(m_2 + \Re k_2)^2-\tau_1^2|\right)^{-1}.
\ee for $\xi_1\in [-\pi, \pi], \tau_1\in 2\pi \N$. If $k_2\in Z$, then $\Re
k_2\in[\theta,\pi+\delta]$, and we clearly have
$$|m_2+\Re k_2 \pm \tau_1|\geq \theta.$$
An elementary argument also shows
$$\min_{m_2\in 2\pi\Z}\{|m_2+\Re k_2 - \tau_1|,|m_2+\Re k_2 + \tau_1|\}\geq \tau_1.$$
Thus
$
|(m_2 + \Re k_2)^2-\tau_1^2|=|m_2 + \Re k_2+\tau_1||m_2 + \Re k_2-\tau_1|\geq \theta\tau_1
$
and
$$
\left(\min_{m_2\in 2\pi\Z}|(m_2 + \Re k_2)^2-\tau_1^2|\right)^{-1} \leq \frac{1}{\theta\tau_1}.
$$
We hence may choose $\tau_1\in 2\pi\N$ so large that for all
$\xi_1 \in [-\pi, \pi]$
\be\nn \|-\Delta_{(\xi_1 + i
\tau_1,k_2)}^{-1}\| \leq \frac{1}{2\lambda
\|\eps_0\|_{\infty}}\qquad(k_2\in Z) \ee holds. The standard Neumann
series argument then shows that $$(-\Delta_{(\xi_1+i \tau_1 ,k_2)}-
\lambda\eps_0)^{-1}$$ exists for $\xi_1\in [-\pi, \pi]$ and $k_2\in Z$.
\end{proof}
As a consequence of this theorem, $k_1 \mapsto T(k_1, k_2)$ is meromorphic in the variable $k_1$
for each $k_2\in Z$ (see proposition \ref{prop1}). The following remark is easy to see:
\begin{remark}
If $k_1$ is a pole of $T(\cdot, k_2)$, then $k_1+2\pi m$ is also a pole of
$T(\cdot, k_2)$ for any $m\in \Z$, i.e. the poles of $T(\cdot, k_2)$ repeat
periodically in real direction with period $2\pi$. Thus, equivalently, we may regard the
poles of $T(\cdot, k_2)$ as elements of $\C/2\pi$.
Moreover, if we define for $r\in L^2(\Omega)$
\be\label{defH}
H(k_1, k_2) r := e^{i k_1 x_1} T(k_1, k_2)[e^{- i k_1 \cdot} r]
\ee
($e^{-i k_1 \cdot}$ means the function $(x_1, x_2)\mapsto e^{-i k_1 x_1}$) then
\be\label{per}
H(k_1+2\pi m, k_2) = H(k_1, k_2)\quad (m\in \Z)
\ee
holds, whenever $k_1$ is not a pole of $T(\cdot, k_2)$.
\end{remark}
Choose a number $$\tau_1\in 2\pi \N$$ with the properties from theorem \ref{freeofpoles}.
This $\tau_1$ will be fixed for the rest of the paper.
Consider the following set in the complex $k_1$-plane:
$$D:=\{z\in\C/2\pi :~|\Im z|< \tau_1 \}.$$
In the following we will study the behavior of the poles
of $T(\cdot,k_2)$ lying in $D$ when $k_2$ varies in $Z$.

\begin{theorem}\label{thmpolesanalytic}
There exists a set discrete $\Ex\subset Z$ not accumulating anywhere in $\cl{Z}$ and
a number $N\in \N_0$ with the following property: the number of poles of $T(\cdot, k_2)$
inside $D$ is equal to $N$ for all $k_2\in Z\setminus \Ex$. Moreover, given any simply connected
domain $\mathcal{U}\subset Z\setminus \Ex$, there exist analytic functions $\{p_j\}_{j=1, \ldots, N}$
defined on $\mathcal{U}$ such that for all $k_2\in\mathcal{U}$,
$$
k_1 \in D~~\text{is a pole of}~~T(\cdot, k_2)~\text{if and only if}~k_1
= p_j(k_2)~\text{for some}~j\in \{1, \ldots, N\}.
$$
\end{theorem}

The proof can be found in the appendix.

Now we introduce the set $Z_0$. Recall the definition of $\cP$ in \eqref{eqP}.

\begin{proposition}\label{propdelta0}
There exists a simply connected open set $$Z_0\subset Z\setminus \Ex$$ in the complex $k_2$-plane,
such that $Z_0\cap\cP$ has positive one-dimensional Lebesgue measure in $[-\pi,\pi]$,
and a $\delta_0>0$ with the following property: if
$k_2$ varies in $Z_0\cap\cP$ then one and only one of the following alternatives
holds for each $p_j$:
\begin{enumerate}
\item[(i)] $p_j(k_2)\in \R$
\item[(ii)] $|\Im p_j(k_2)| > 2\delta_0$.
\end{enumerate}
Here $\{p_j\}_{j=1,\ldots N}$ is the collection of analytic functions from theorem \ref{thmpolesanalytic} defined
on $Z_0$.
\end{proposition}
This proposition means: if $k_2$ varies in $Z_0\cap\cP$, each pole of $T(\cdot,k_2)$ either stays on
the real axis or it keeps a
distance greater than $2\delta_0$ from the real axis.

\begin{proof}
We sketch the easy proof. Choose any $\kappa^{*}\in Z\cap \cP\setminus \Ex$ and a small
ball $\til{Z_0}\subset Z$ around it, which does not contain any point of $\Ex$
and such that $\widetilde{Z_0}\cap \cP$ has positive Lebesgue measure.
By theorem \ref{thmpolesanalytic} the poles can be represented by
analytic functions $p_j$ on $\til{Z_0}$. If $p_j(\kappa^*)\in \R$,
then by power series expansion we see that either $p_j(k_2)$ must be real for all real $k_2$
close to $\kappa^*$ or $p_j(k_2)\notin \R$ for all real $k_2\neq \kappa^*$ close to $\kappa^*$. On the other hand,
if $p_j(\kappa^*)\notin\R$ then for all $k_2$ close to $\kappa^*$ we also have $p_j(k_2)\notin \R$.
Now take any ball $Z_0\subset \til{Z_0}$ sufficiently close to $\kappa^*$
such that $Z_0\cap \cP$ has positive one-dimensional Lebesgue measure, but $\kappa^*\notin Z_0$
and the previous reasoning applies.
Now we are in the situation that some of the $p_j(k_2)$ are real for all $k_2\in Z_0\cap \cP$ and
the other poles have a positive distance to the real axis (or one of the cases occurs exclusively).
\end{proof}

\begin{figure}[htbp]
\caption{Illustration of the $k_2$-plane and the sets $\til{Z_0}, Z_0$
used in the proof of proposition \ref{propdelta0}. The black points indicate elements in the 
set $\Ex$.
\label{figZZ0near}}
\begin{center}
\begin{picture}(0,0)%
\includegraphics[scale=0.4]{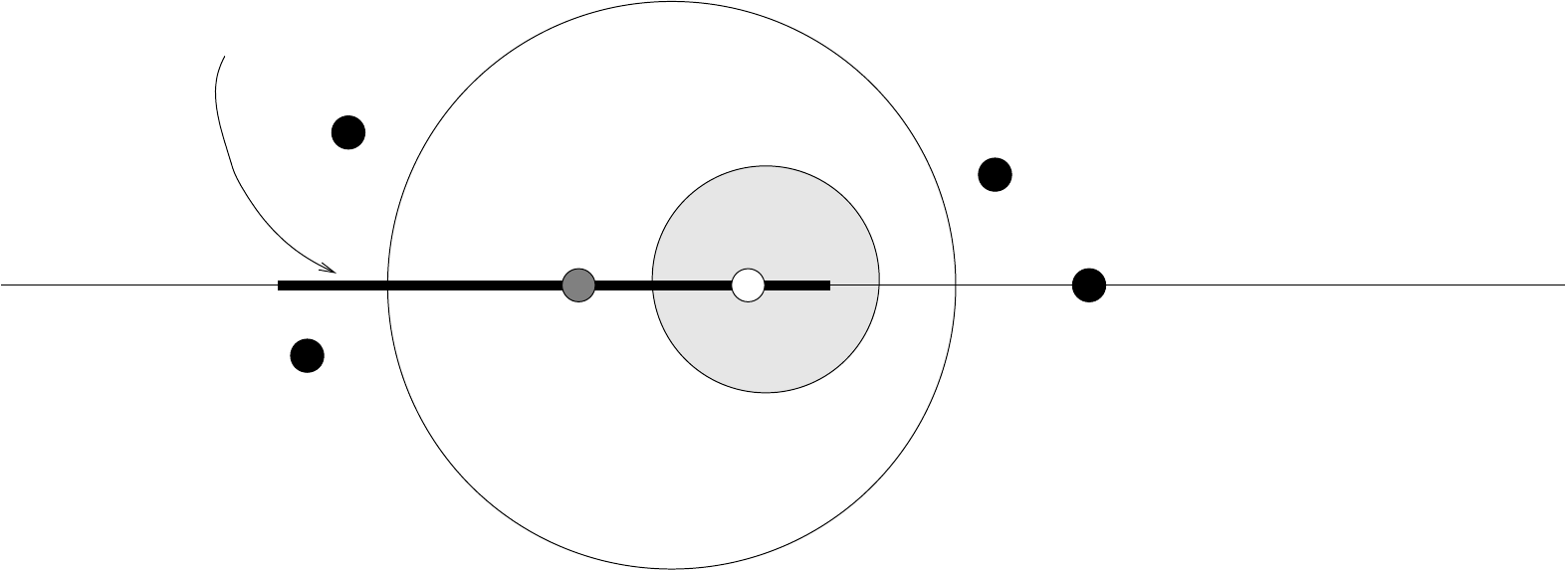}%
\end{picture}%
\setlength{\unitlength}{4144sp}%
\begingroup\makeatletter\ifx\SetFigFont\undefined%
\gdef\SetFigFont#1#2#3#4#5{%
  \reset@font\fontsize{#1}{#2pt}%
  \fontfamily{#3}\fontseries{#4}\fontshape{#5}%
  \selectfont}%
\fi\endgroup%
\put(2400,1000){\text{$Z_0$}}
\put(1800,1300){\text{$\til{Z_0}$}}
\put(1700,630){\text{$\kappa^*$}}
\begin{picture}(13000,1500)(52,956)
\put(700,2600){\text{$\mathcal{P}$}}
\end{picture}%
\end{center}
\end{figure}
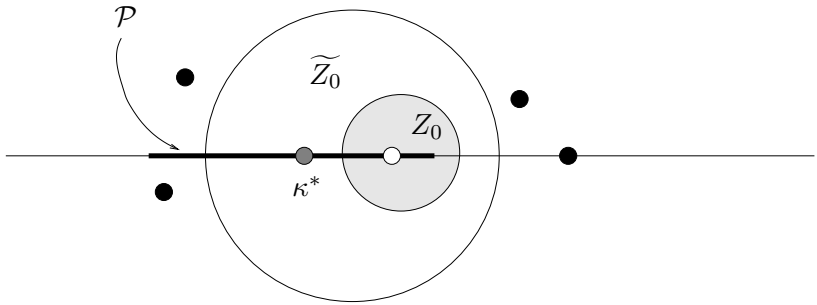

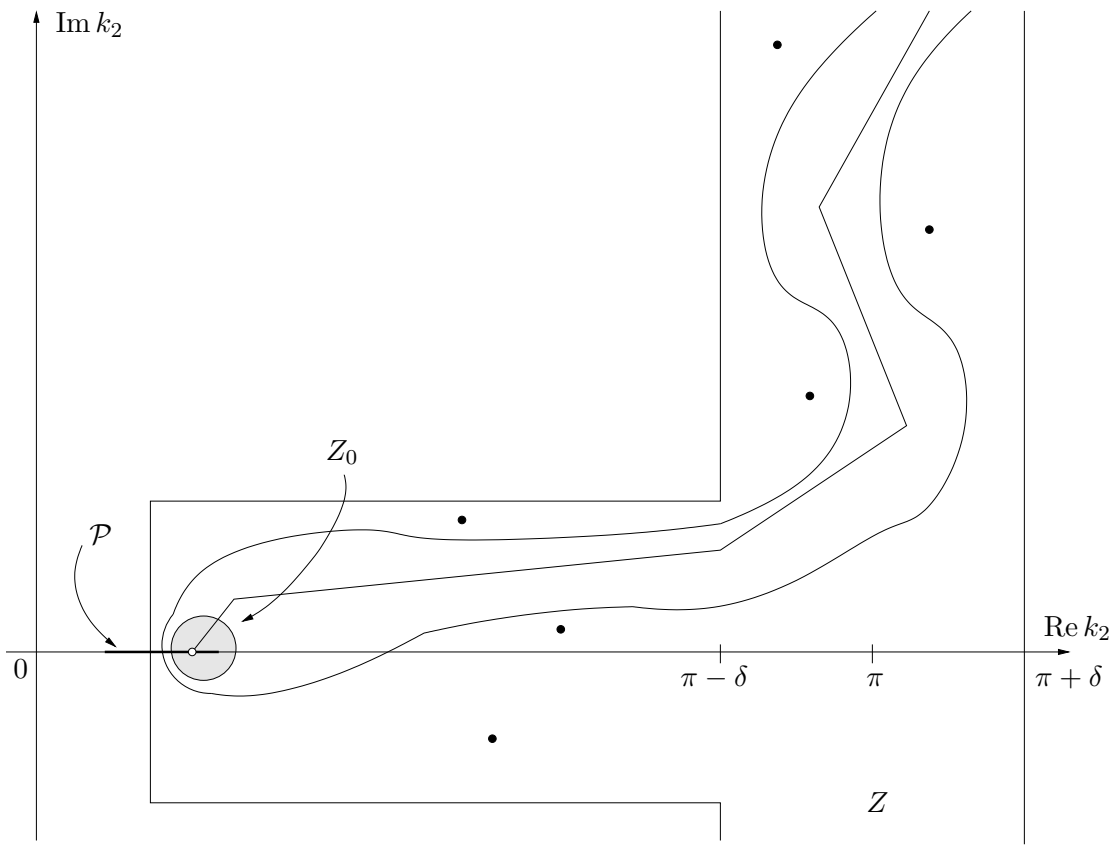
\begin{figure}[htbp]
\caption{Illustration of the path $\Gamma$ and its neighborhood $\Ne(\Gamma)$
in the complex $k_2$-plane. The black points indicate
elements of the exceptional set $\Ex$, the white point is $\Gamma(0)$.\label{figZZ0}}
\begin{center}
\begin{picture}(0,0)%
\includegraphics[scale=0.5]{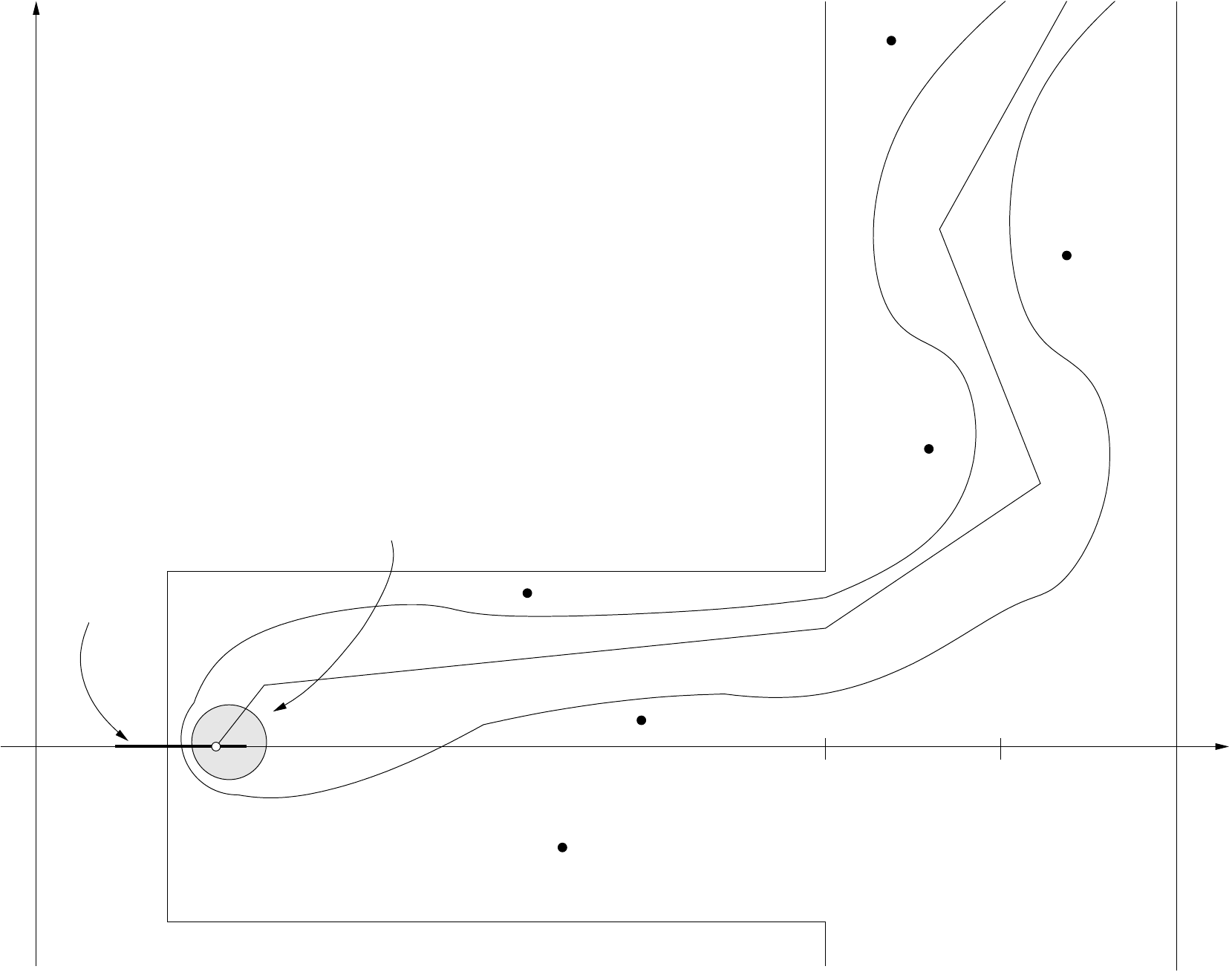}%
\end{picture}%
\setlength{\unitlength}{4144sp}%
\begingroup\makeatletter\ifx\SetFigFont\undefined%
\gdef\SetFigFont#1#2#3#4#5{%
  \reset@font\fontsize{#1}{#2pt}%
  \fontfamily{#3}\fontseries{#4}\fontshape{#5}%
  \selectfont}%
\fi\endgroup%
\put(500, 1800){\text{$\mathcal{P}$}}
\put(1900, 2300){\text{$Z_0$}}
\put(4000, 950){\text{$\pi-\delta$}}
\put(5100, 950){\text{$\pi$}}
\put(6100, 950){\text{$\pi+\delta$}}
\put(6150, 1250){\text{$\Re k_2$}}
\put(5100, 200){\text{$Z$}}
\put(50, 1000){\text{$0$}}
\put(300, 4850){\text{$\Im k_2$}}
\begin{picture}(13000,5000)(52,956)
\end{picture}%
\end{center}
\end{figure}

\subsection{Derivation of a Fredholm problem on $\Omega$}
Recall that the partial Floquet-Bloch transform $V_{x_1}$  in $x_1$-direction is given by \eqref{Vx1}. 
\begin{proposition}\label{proponw}
Suppose $k_2\in Z_0\cap \cP$ is fixed and $0\neq g\in H^2_{\text{per}}(S)$ solves
\be\label{eq7}
(-\Delta_{k_2}-\lambda\eps(\bx))g=0 \quad~\text{on}~~S.
\ee
Then there exists an analytic function $k_1\mapsto w(\cdot, k_1)\in L^2(\Omega)$ defined
for
$$
k_1 \in \mathcal{O} := \{ z \in \C : |\Im z| < 2\delta_0 \}
$$
such that $(V_{x_1} g)(\cdot, k_1) = w(\cdot, k_1)$ for
almost all $k_1\in [-\pi, \pi]$ and
\be\label{eq13}
w(\bx,k_1)-2 \pi \lambda (T(k_1,k_2)V_{x_1}\eps_1 V_{x_1}^{-1}w)(\bx,k_1)=0\quad (\bx \in \Omega)
\ee
for all $k_1 \in \mathcal{O}$.
\end{proposition}
\begin{proof}
Let $\til{w}(\cdot,k_1)\in L^2(\Om)$ be the Floquet transform of $g$ in $x_1$-direction:
\be \label{eq9}
\til{w}(\cdot,k_1)=(V_{x_1}g)(\cdot,k_1).
\ee
By applying the isometry $V_{x_1}$ to \eqref{eq7} and recalling the periodicity of $\eps_0$
in $x_1$-direction, we see that $\til{w}$ solves
\be\nn
(-\Delta_{(k_1, k_2)}-\lambda \eps_0) \til{w}(\cdot, k_1) - \lambda V_{x_1}(\eps_1 g)(\cdot, k_1) = 0.
\ee
for almost all $k_1\in[-\pi,\pi]$. $(-\Delta_{(k_1, k_2)}-\lambda \eps_0)^{-1}=2\pi T(k_1, k_2)$
exists for all $k_1$ except discretely many points. At all points $k_1$, where $T(k_1, k_2)$ exists,
\be\label{eq8}
\til{w}(\cdot,k_1)- 2\pi\lambda (T(k_1,k_2)V_{x_1}\eps_1 V_{x_1}^{-1}\til{w})(\cdot,k_1)=0
\ee
holds.
Since we want to arrive at \eqref{eq13}, \eqref{eq8} suggests to define $w$ by
\be\label{eq10}
w(\cdot,k_1):= 2\pi\lambda (T(k_1,k_2)V_{x_1}\eps_1 V_{x_1}^{-1}\tilde w)(\cdot,k_1).
\ee
By \eqref{suppeps}, $\eps_1$ has compact support in $x_1$-direction, and so
$(V_{x_1}\eps_1 V_{x_1}^{-1}\til w)(\cdot,k_1)$ is analytic in the variable $k_1$ on $\C$ with
values in $L^2(\Om)$ (see \eqref{Vcomp}). $k_1\mapsto T(k_1,k_2)$ is a meromorphic
operator valued function with at most finitely many poles in $\mathcal{O}$. By proposition \ref{propdelta0},
these poles lie on the real axis.
So the right-hand side of \eqref{eq10} makes sense for $k_1\in \mathcal{O}$,
except when $k_1$ is a pole of $T(\cdot, k_2)$.

We claim that the poles of $w$ can actually be removed by continuity, and thus $w$ is
analytic on the whole of $\mathcal{O}$.

By the relation \eqref{eq9}, we have $\til{w} = w$ a.e. on the real axis.
But $\til{w}(\cdot, k_1)$ is square integrable on $[-\pi,\pi]$ with respect to
$k_1$, so an elementary argument shows $\til{w}$ cannot have poles in $[-\pi, \pi]$.
This proves that $w$ is analytic on the whole of $\mathcal{O}$.
\end{proof}

\begin{proposition}\label{propeqonOmega} Let $k_2\in Z_0\cap\cP$ be fixed.
Any $g\neq0$ solving \eqref{eq7} solves
\be\label{eq12}
g|_\Omega-\lambda A(k_2)\eps_1 g|_\Omega=0,
\ee
where $A(k_2)$ is defined by
\be\label{eqdefA}
A(k_2)r=\int_{[-\pi,\pi]+i\delta_0}H(k_1,k_2)r~dk_1
\ee
with $H(k_1,k_2)$ defined by \eqref{defH} and $\delta_0$ from proposition \ref{propdelta0}.
Moreover, $g|_\Omega\neq 0$.
\end{proposition}
\begin{proof}
First we prove that $g|_\Omega\neq 0$; thus, assume the contrary.
Using the fact that $g$ solves
\eqref{eq7} and a unique continuation principle (see
\cite{Protter}), we conclude $g \equiv 0$ on the whole of $S$, a
contradiction.

Now we show that \eqref{eq12} holds. With $w(\cdot,k_1)=V_{x_1}g$, we get from \eqref{eq13}
\be\nn
V_{x_1}^{-1}[w-2\pi \lambda T(\cdot,k_2)(V_{x_1}\eps_1 V_{x_1}^{-1}w)] = 0.
\ee
The inverse Floquet transform of the term containing $T(\cdot,k_2)$ can be rewritten as
$$
\lambda \int_{[-\pi, \pi]} H(k_1, k_2)(\eps_1 g)~d k_1,
$$
where we used $(V_{x_1} \eps_1 g)(k_1, \bx) = (2\pi)^{-\frac{1}{2}} e^{-i k_1 x_1} \eps_1(\bx) g(\bx)$
(see \eqref{Vcomp}).
Observe carefully that from \eqref{eq13} we know that the integrand $k_1\mapsto H(k_1, k_2)(\eps_1 g)$ is analytic
on the set $\mathcal{O}$ from proposition \ref{proponw}.
Thus we may use Cauchy's integral theorem to deform
the integral over $[-\pi, \pi]$ into an integral over $[-\pi, \pi]+i\delta_0$ (lateral
contributions cancel due to the periodicity \eqref{per}):
\be
\nn \lambda \int_{[-\pi,\pi]+i\delta_0}H(k_1, k_2)(\eps_1 g)~dk_1.
\ee
Hence we arrive at the following problem for $g$ posed on $\Om$:
\be\label{eq121}\nn
g|_\Omega-\lambda A(k_2)\eps_1 g|_\Omega =0,
\ee
where $A(k_2):L^2(\Om)\to L^2(\Om)$ is as in the statement of the proposition.
\end{proof}

Using proposition \ref{propeqonOmega} we derive our final Fredholm equation.
Since $\til{v}=\til{v}(\cdot,k_2)$ solves \eqref{eq71} we have that $v:=\til{v}|_\Omega$ solves
\be\label{eqfinalfred}
v-\lambda A(k_2)\eps_1v=0\quad~~\text{on}~\Omega.
\ee
Notice that if $k_2\in Z_0\cap\cP$ then $\til{v}(\cdot,k_2)\neq 0$ and $v\neq 0$.

Equation \eqref{eqfinalfred} is crucial to prove our main result. In the following section
we will show that the operator $A(k_2)$ is compact and can be defined for $k_2$ in a
certain region of the complex $k_2$-plane
with unbounded imaginary part. Moreover we will see that in that region $k_2\mapsto A(k_2)$ is an analytic
operator-valued function.

\section{Analytic Continuation.}\label{sec_cont}

\subsection{Construction of the analytic continuation}
We will now describe the analytic continuation of the operator $A(k_2)$
to values $k_2$ with large imaginary part, along a certain path
$\Gamma$ lying in the set $Z$ in the complex $k_2$-plane.
\begin{lemma}
There exist a continuous path $$\Gamma : [0,\infty) \to Z\setminus\Ex$$ satisfying
\begin{enumerate}
\item[(i)] $\Gamma(0)\in Z_0\cap\cP$,
\item[(ii)] $t\mapsto \Im \Gamma(t)$ is nondecreasing,
\item[(iii)] $\Im\Gamma(t)\to +\infty$ for $t\to \infty$,
\end{enumerate}
with the property that there exists a simply connected neighborhood
$$\Ne(\Gamma)\subset Z\setminus \Ex$$ of the path $\Gamma$ containing $Z_0$ and a
$N\in \N$ such that the number of poles of $T(\cdot, k_2)$ in $D$ is
equal to $N$ for all $k_2\in \Ne(\Gamma)$. Moreover, there exists a
collection of analytic functions $\{q_j\}_{j=1}^N$,
$$q_j : \Ne(\Gamma) \to \C/2\pi$$
such that the poles of $T(\cdot, k_2)$ in $D$ are exactly given by
$$
q_j(k_2)\quad(j=1, \ldots, N).
$$
\end{lemma}
Notice that $\Ne(\Gamma)\cap \cP$ has positive one-dimensional Lebesgue measure.

\begin{proof}
According to theorem \ref{thmpolesanalytic}, the number of poles of $T(\cdot, k_2)$ is
equal to some number $N\in \N$,
as $k_2$ varies in $Z\setminus \Ex$. Since $\Ex$
does not accumulate anywhere in $\overline{Z}$,
it is clear that one can choose a continuous path $\Gamma$ with the
properties in the statement of the lemma. The sets $\Ex$
and $\Gamma([0,\infty)) \subset \C$ are closed.
Hence by the separation properties of the metric space $\C$, there exists
a simply connected neighborhood $\Ne(\Gamma)$ with $\Ex \cap \Ne(\Gamma)=\emptyset$
(see figure \ref{figZZ0} for an illustration). In order to choose the functions
$\{q_j\}$, apply theorem \ref{thmpolesanalytic}.
\end{proof}

Note that in general, the poles are algebraic functions of $k_2$,
i.e. they may behave like complex roots in the vicinity of points of
the exceptional set $\mathcal{E}$.
Here, we get analyticity by avoiding
the exceptional set $\Ex$.

The collection $\{q_j\}_{j=1}^N$ can be written as a disjoint union
$$
\{q_j^+\}_{j=1}^{N^+} \cup \{q_j^-\}_{j=1}^{N^-}
$$
with $N^+ + N^- = N$ and such that
$$
\Im q_j^+(\Gamma(0)) >  \delta_0, ~~~\Im q_j^-(\Gamma(0)) \leq 0
\quad (j=1,\ldots, N^\pm)
$$
with $\delta_0$ from proposition \ref{propdelta0}.

\begin{proposition}\label{propcoin1} For any $k_2\in Z_0\cap \cP$ we have the following representation of $A(k_2)$:
\ben
A(k_2) r &=& \int_{[-\pi,\pi]+i \tau_1} H(k_1, k_2) r~dk_1
+ 2\pi i \sum_{j=1}^{N^+} \res(H(\cdot, k_2) r, q_j^+(k_2)) \label{formA}
\een
for all $r\in L^2(\Omega)$.
In the formula \eqref{formA}, both sides are understood as a functions in $L^2(\Omega)$, and
$$\res(H(\cdot, k_2) r, q_j^+(k_2))$$
denotes the residue of the meromorphic $L^2(\Omega)$-valued
function $k_1\mapsto H(k_1, k_2) r$ at the pole $q_j^+(k_2)$.
Moreover, the right-hand side of
\eqref{formA} makes sense for all $k_2\in \Ne(\Gamma)$ and defines a
continuation of $A(k_2)$ to all of $\Ne(\Gamma)$. We use the same symbol $A(k_2)$
to denote the original operator for $k_2\in Z_0\cap \cP$ and its continuation defined
for $k_2\in \Ne(\Gamma)$.
\end{proposition}

\begin{proof}
Choose a contour $\gamma=\gamma(k_2)$ in the complex plane as indicated
in figure \ref{figgamma}, where the lateral parts of $\gamma$ avoid poles $q_j^+(k_2)$ with
real part equal to $-\pi$ or $\pi$. The lateral part to the right
has the same shape as the left part, but it is shifted by $2\pi$
in positive real direction.

None of the $q_j^+(k_2)$ lies on the real axis for $k_2\in Z_0\cap\cP$. Using the residue theorem, we get
\be\label{eq5}
\int_{\gamma} H(k_1, k_2) r~dk_1 = 2\pi i \sum_{j=1}^{N^+} \res(H(\cdot, k_2) r, q_j^+(k_2)),
\ee
since the $q_j^+(k_2)$ are exactly the poles between the line segments $[-\pi,\pi]+i\delta_0$ and
$[-\pi,\pi]+i\tau_1$ in the upper half-plane for $k_2\in Z_0\cap\cP$.
Since the contributions from the lateral parts of the contour cancel due to the periodicity
\eqref{per}, the left-hand side of \eqref{eq5} is just
\bea
\int_{[-\pi,\pi]+i\delta_0} H(k_1, k_2) r~dk_1 - \int_{[-\pi, \pi]+i\tau_1} H(k_1, k_2) r~dk_1 =
A(k_2)r - \int_{[-\pi, \pi]+i\tau_1} H(k_1, k_2) r~dk_1.
\eea
This proves the proposition.
\end{proof}

\begin{figure}[htbp]
\caption{\label{figgamma} The contours $\gamma$ and $\gamma_j$ used in the
proof of lemma \ref{lemAanalytic} and proposition \ref{propcoin1}}
\begin{center}
\setlength{\unitlength}{2400sp}%
\begingroup\makeatletter\ifx\SetFigFont\undefined%
\gdef\SetFigFont#1#2#3#4#5{%
  \reset@font\fontsize{#1}{#2pt}%
  \fontfamily{#3}\fontseries{#4}\fontshape{#5}%
  \selectfont}%
\fi\endgroup%
\begin{picture}(8528,8098)(461,-7698)
\linethickness{0.02cm}
\put(473,-7020){\vector( 1, 0){8504}}
\put(4717,-7665){\vector( 0, 1){7500}}
\linethickness{0.06cm}
\put(946,-750){\line( 1, 0){7559}}
\put(946,-6720){\line( 1, 0){7559}}
\put(550, -7420){$-\pi$}
\put(8500, -7420){$\pi$}
\put(4500, -7420){$0$}
\put(4400, -6600){$\delta_0$}
\put(4400, -650){$\tau_1$}
\put(8800, -750){$\gamma$}
\put(946, -750){\line( 0,-1){3000}}
\put(946, -3750){\line( -1, 0){500}}
\put(946, -4750){\line( -1, 0){500}}
\put(446, -3740){\line( 0, -1){1000}}
\put(946, -4750){\line( 0, -1){1990}}
\put(8520, -750){\line( 0,-1){3000}}
\put(8520, -3750){\line( -1, 0){500}}
\put(8520, -4750){\line( -1, 0){500}}
\put(8020, -3740){\line( 0, -1){1000}}
\put(8520, -4750){\line( 0, -1){1990}}
\put(8520, -4400){\circle*{100}}
\put(946, -4400){\circle*{100}}
\put(6000, -2560){\circle*{100}}
\put(6000, -2560){\circle{4000}}
\put(6600, -2200){\text{$\gamma_j$}}
\put(5700, -2850){\text{\tiny $q_j^+(k_2^0)$}}
\put(4800, -5400){\circle*{100}}
\put(3000, -4600){\circle*{100}}
\put(3400, -2200){\circle*{100}}
\put(9000, -7300){\text{$\Re k_1$}}
\put(3900, -50){\text{$\Im k_1$}}
\end{picture}
\end{center}
\end{figure}
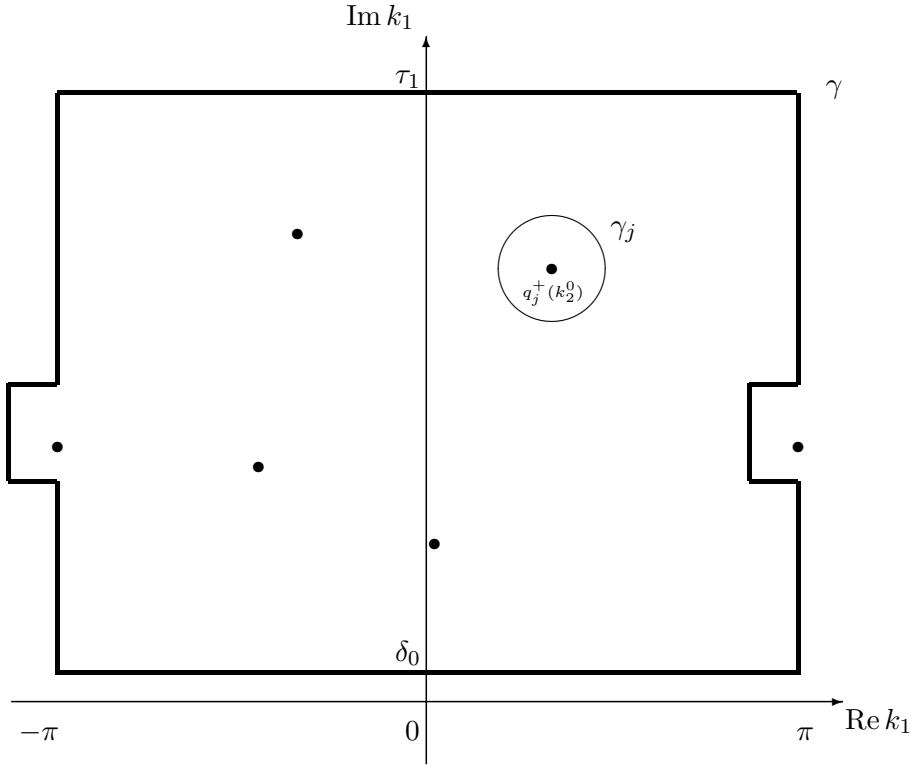

\begin{lemma}\label{lemAanalytic}
$A(k_2): L^2(\Omega)\to L^2(\Omega)$ is a compact operator for each $k_2\in \Ne(\Gamma)$;
moreover, $k_2 \mapsto A(k_2)$ is an analytic operator-valued function.
\end{lemma}
\begin{proof}
The compactness is implied by the standard estimate
\be\label{sr}\nn
\|\nabla T(k_1, k_2) f\|_{L^2(\Omega)} \leq C(k_1, k_2)\left[
\|f\|_{L^2(\Omega)} + \|T(k_1, k_2) f\|_{L^2(\Omega)}
\right].
\ee
The integral over $[-\pi, \pi]+i\tau_1$ in \eqref{formA}
depends analytically on $k_2$, since $H(k_1, k_2)$ exists for all $k_2\in \Ne(\Gamma)$
and depends analytically on $k_2$. In order to prove the analyticity of the sum in
\eqref{formA}, fix a $k_2^0\in \Ne(\Gamma)$ and choose a system of small circles
$\gamma_j$ in the complex $k_1$-plane, each of the $\gamma_j$ enclosing one of
the $q^+_j(k_2^0)$ and no other poles. Since the
number of poles stays constant away from the set $\Ex$,
each of the $\gamma_j$ encloses exactly the pole $q_j^+(k_2)$
for $k_2$ in a small neighborhood of $k_2^0$. Hence the sum in \eqref{formA} may be written as
$$
\sum_{j=1}^{N^+} \oint_{\gamma_j} H(k_1, k_2) r ~dk_1
$$
for all $k_2$ close to $k_2^0$ and we see that it is obviously analytic in $k_2$.
\end{proof}

\subsection{Discussion of the analytic continuation process}
Since the analytic continuation described above is rather difficult,
we now make a few more informal remarks in order to make the
construction of the operator family $A(k_2)$ more accessible.

Starting point is the relation
$$
A(k_2)r = \int_{[-\pi,\pi]+i\delta_0} H(k_1, k_2) r ~d k_1
$$
which holds for $k_2=\Gamma(0)\in Z_0\cap \cP $. First we deform
the integral over $[-\pi, \pi]+i\delta_0$ into an integral over $\gamma(k_2)$ as in figure \ref{figgamma},
thereby obtaining
$$
A(k_2)r = \int_{[-\pi,\pi]+i\tau_1} H(k_1, k_2) r ~d k_1
+ 2\pi i \sum_{j=1}^{N^+} \res(H(\cdot, k_2) r, q_j^+(k_2)).
$$
Note that only the poles $q_j^+(k_2)=q_j^+(\Gamma(0))$ appear, since only those
lie between the line segments $[-\pi,\pi]+i\delta_0$ and $[-\pi,\pi]+i\tau_1$ in the upper half plane.
Now as we let $k_2$ move along the path $\Gamma$ to values
with large imaginary part, the poles of $T(\cdot,k_2)$ will move. Note carefully that the total number
of poles is constant along the path (and also for $k_2$ in the neighborhood $\Ne(\Gamma)$)
and none of the poles $q_j^+(k_2)$ collides with another $q_i^+(k_2)$ nor
with another $q_i^-(k_2)$. Note, however, that the poles $q_j^+(k_2)$ may also cross
the line segment $[-\pi, \pi]+i\delta_0$ when $k_2$ moves along the path $\Gamma$.

The point in defining the operator family $A(k_2)$ as we did above is to ensure \emph{analyticity}
in $k_2$. To this end, we avoided the exceptional set $\Ex$ in the construction of
the path $\Gamma$, which has the consequence that each summand in the sum over the residues
in \eqref{formA} is analytic in $k_2$.

Upon closer inspection of the formula \eqref{formA}, the reader might ask why we insist
on analyticity of each term in the sum. Whenever two poles lie close to each other, we might
have replaced the sum of their residues by corresponding contour integrals, thereby
obtaining analyticity in spite of the collision of poles, thereby choosing the path $\Gamma$
without avoiding $\Ex$. Actually, this does not work.

Whenever two poles originating
from between the line segments $[-\pi,\pi]+i\delta_0$ and $[-\pi,\pi]+i\tau_1$ in the upper-half plane,
say $q_j^+$ and $q_i^+$, collide, the merging of the two of them
into one contour integral helps. However, we would
have to carefully pay attention to the possibility that a pole $q_j^+$ moves around and
collides with a pole $q_l^-$. In this situation a merging of the two destroys analyticity.

The poles $q_j^+(k_2), q_i^-(k_2)$ will eventually become neatly
separated if $k_2$ moves along a sequence with $\Im k_2 \to +\infty$, which is the content of section \ref{sec_asymp}.
But for ''intermediate'' $k_2$, we have no information on the movement of the poles.
Therefore, it is a good idea to avoid all the collisions beforehand and just to follow the poles
$q_j^+$. And fortunately avoiding collisions is a fairly easy task.

\section{Asymptotic behavior of $A(k_2)$ as $\Im k_2 \to \infty$}\label{sec_asymp}

\subsection{Asymptotic localization of the poles of $T(\cdot, k_2)$}

We now define a set $\Lines$ in the complex $k_1$-plane, on which we have good control over inverse
of the shifted Laplacian $-\Delta_{\bk}^{-1}$.
\begin{definition}\label{defLines} We define the set $\Lines\subseteq\C$ consisting of four vertical and a
finite family of horizontal lines in the complex $k_1$-plane (see figure \ref{figLines}) by
\beq\label{eqlines}\nn
\Lines:=
 \left((\pm\frac{\pi}{2}\pm 2\delta)+i\R\right)\cup
 \left(\bigcup_{\nu\in 2\pi\Z,~ |\nu|\leq \tau_1}[-\pi,\pi]+i \nu\right).
\eeq
\end{definition}
In lemma \ref{lemsymbolest} (see appendix), we derive an important estimate for $-\Delta_{(k_1, k_2)}^{-1}$ for
$k_1\in \Lines$ and $k_2$ with large imaginary part.

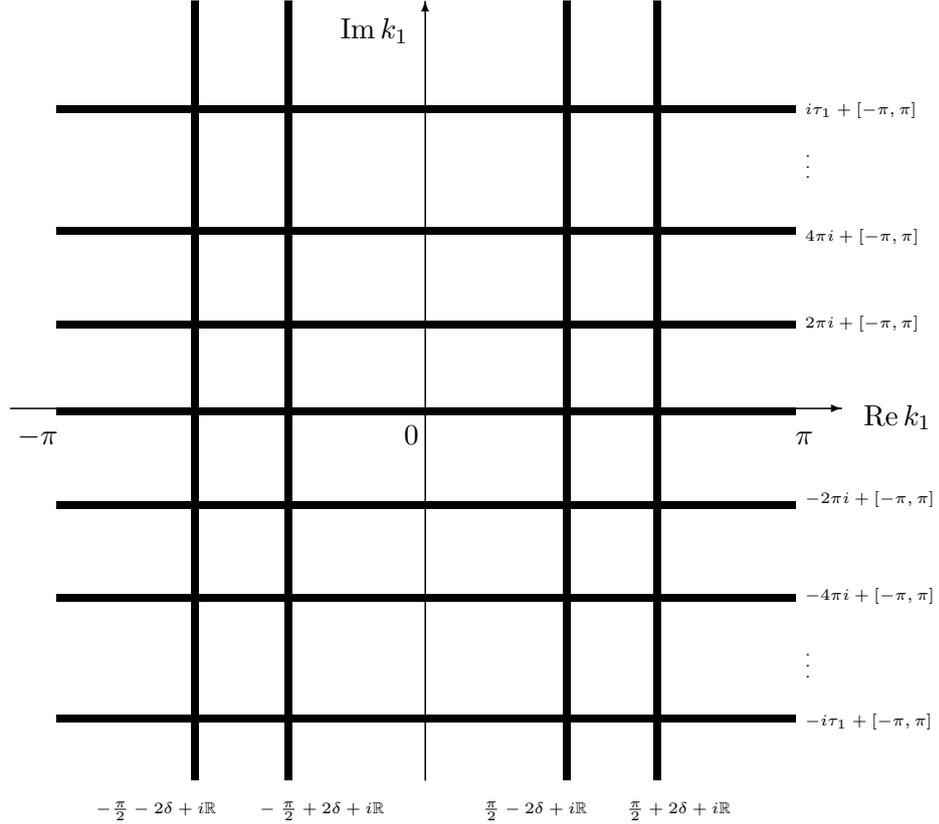
\begin{figure}[htbp]
\begin{center}
\setlength{\unitlength}{2400sp}%
\begingroup\makeatletter\ifx\SetFigFont\undefined%
\gdef\SetFigFont#1#2#3#4#5{%
  \reset@font\fontsize{#1}{#2pt}%
  \fontfamily{#3}\fontseries{#4}\fontshape{#5}%
  \selectfont}%
\fi\endgroup%
\caption{\label{figLines} Illustration of the set $\Lines$ in the complex plane.}
\begin{picture}(8528,8098)(461,-7698)
\linethickness{0.02cm} { \put(473,-3830){\vector( 1, 0){8504}}
}%
{ \put(4717,-7665){\vector( 0, 1){8032}}
}%
\linethickness{0.1cm} { \put(946,-2963){\line( 1, 0){7559}}
}%
{ \put(946,-1996){\line( 1, 0){7559}}
}%
{ \put(946,-750){\line( 1, 0){7559}}
}%
{ \put(946,-4830){\line( 1, 0){7559}}
}%
{ \put(946,-5782){\line( 1, 0){7559}}
}%
{ \put(946,-7020){\line( 1, 0){7559}}
}%
{ \put(2363,367){\line( 0,-1){8032}}
}%
{ \put(3315,367){\line( 0,-1){8032}}
}%
{ \put(6165,367){\line( 0,-1){8032}}
}%
{ \put(7088,367){\line( 0,-1){8032}}
}%
{ \put(946,-3862){\line( 1, 0){7559}}
\put(550, -4200){$-\pi$}
\put(1350, -8000){\tiny $-\frac{\pi}{2}-2\delta+i\R\qquad-\frac{\pi}{2}+2\delta+i\R$}
\put(5300, -8000){\tiny $\frac{\pi}{2}-2\delta+i\R\qquad\frac{\pi}{2}+2\delta+i\R$}
\put(8500, -4200){$\pi$}
\put(8600, -800){\tiny $i \tau_1  +[-\pi, \pi]$}
\put(8600, -1460){\tiny $\vdots$}
\put(8600, -2100){\tiny $4\pi i +[-\pi, \pi]$}
\put(8600, -3000){\tiny $2\pi i + [-\pi, \pi]$}
\put(8600, -4800){\tiny $-2\pi i + [-\pi, \pi]$}
\put(8600, -5800){\tiny $-4\pi i+ [-\pi, \pi]$}
\put(8600, -7100){\tiny $-i \tau_1  +[-\pi, \pi]$}
\put(8600, -6600){\tiny $\vdots$}
\put(4500, -4200){$0$}
\put(9200, -4000){\text{$\Re k_1$}}
\put(3850, -20){\text{$\Im k_1$}}
}%
\end{picture}
\end{center}
\end{figure}

\begin{definition} For $m_2\in 2\pi \Z$, $-\tau_1+2\pi\le m_2\le \tau_1$ we define the following rectangular
contours in the complex
$k_1$-plane:
$$\Gamma^{\pm}_{m_2} : = \left\{k_1\in\C:~\max\left\{\left|\frac{\Re k_1\mp\frac{\pi}{2}}{2\delta}\right|,
\left|\frac{\Im k_1 - m_2 + \pi}{\pi}\right|\right\}=1\right\}.$$
\end{definition}

The most important features of these contours are $\Gamma^{\pm}_{m_2}\subseteq \Lines$ (see
definition \ref{defLines} and figure \ref{figLines}) and the following

\begin{lemma}
There exists a number $M=M(\delta,\tau_1,\lambda)>0$
such that for all $k_2\in \Ne(\Gamma)$ with $\Im k_2 = \frac{\pi}{2}+\ell,
\ell\in 2\pi \N, \ell > M$, each of the $q_j^+(k_2)$
is enclosed by one of the contours $\Gamma_{m_2}^{\pm}$.
Moreover, each contour encloses exactly one of
the $q_j^+(k_2)$ and no other pole of $T(\cdot, k_2)$.
\end{lemma}

\begin{proof}
First, let $W_\mu(k_2) : D(W_\mu(k_2)) =
H^2_\text{per}(\Omega)\times H^1_{\text{per}}(\Omega)
 \to H^1_{\text{per}}(\Omega)\times L^2(\Omega)$
be defined by \be\label{defW}\nn W_\mu(k_2)(u, v) :=
(v,\Delta_{(0,k_2)}u + 2 i \partial_1 v + \mu \eps_0 u).
\ee
We regard $W_\mu(k_2)$ as an unbounded operator in the
Hilbert space $H^1_{\text{per}}(\Omega)\times L^2(\Omega)$.
It is easy to show that $(-\Delta_{(k_1, k_2)}- \mu \eps_0)^{-1}$ exists if and only if
$(W_\mu(k_2) - k_1)^{-1}$ exists, and as a consequence, $k_1$ is a
pole of $T(\cdot, k_2)$ if and only if $k_1$ is an eigenvalue of
$W_\mu(k_2)$. Moreover, a calculation shows
\beq\nn \ker (W_\mu(k_2)-k_1) = \{
(u, k_1 u) : u\in \ker (-\Delta_{(k_1, k_2)}-\mu \eps_0) \}.
\eeq
From lemma \ref{lemmanullstellen} we see that,
if $\Im k_2 = \frac{\pi}{2}+\ell$, each of the contours $\Gamma^\pm_{m_2}$ encloses exactly one
pole of $-\Delta_{(\cdot, k_2)}^{-1}$, i.e. each $\Gamma^\pm_{m_2}$
encloses exactly one eigenvalue of $W_0(k_2)$. Moreover, from lemma \ref{lemmanullstellen} follows
that the dimension of the corresponding eigenspace is one. Then, since $\Gamma^\pm_{m_2}\subseteq \Lines$,
by \eqref{eqlaplacenorm} (in lemma \ref{lemsymbolest})
there exists a $M=M(\delta,\tau_1,\lambda)$ such that for  $\ell > M$
the norm of $\lambda (-\Delta_{(k_1, k_2)})^{-1} \eps_0$ is less than $1$
on all the contours $\Gamma^\pm_{m_2}$.
By a Neumann series argument, $(-\Delta_{(k_1, k_2)}-\mu \eps_0)^{-1}$ exists for all $\mu \in [0, \lambda]$ on
the contours and hence the dimension of the range of the Riesz projection
$$
\frac{1}{2\pi i}\oint_{\Gamma^\pm_{m_2}} (W_\mu(k_2) - k)^{-1} dk
$$
does not change when $\mu$ varies in $[0, \lambda]$. But this implies that each of the $\Gamma^\pm_{m_2}$
encloses exactly one eigenvalue of $W_\mu(k_2)$ and thus exactly one pole of $T(\cdot, k_2)$.
In particular, each of the $q_j^+(k_2)$ must be enclosed by
exactly one of the contours.
\end{proof}

\subsection{Estimate for $\Im k_2$ large}
The norm of the operator $A(k_2)$ converges to zero when $k_2$ moves along a certain sequence with 
$\Im k_2 \to \infty$.
\begin{theorem} \label{large} There exist constants $C=C(\delta, \tau_1,\lambda)>0,
M=M(\delta,\tau_1,\lambda)>0$ such that for
$k_2\in \Ne(\Gamma)$ of the form $k_2 = \Re k_2 + i(\frac{\pi}{2} + \ell)$ with $\ell \in 2\pi \N, \ell > M$,
\beq\label{est2}\nn
\|A(k_2)\| \leq C \ell^{-1}.
\eeq
\end{theorem}

\begin{proof} Let $k_1\in [-\pi, \pi]+i \tau_1$ or $k_1\in \Gamma^\pm_{m_2}$.
Since $([-\pi, \pi]+i \tau_1) \cup \Gamma^\pm_{m_2}\subseteq \Lines$, by corollary \ref{corTbound}
there exists a $C=C(\delta,\tau_1,\lambda)$ and
a $M=M(\delta,\tau_1,\lambda)$ such that
$$\|T(k_1, k_2)\|\leq C/\ell$$
 for all $\ell>M$. This gives
\be
\|H(k_1, k_2)r\|_{L^2(\Omega)} &\leq& \| e^{-i k_1 \cdot} T(k_1, k_2)(e^{-i k_1\cdot} r)\|_{L^2(\Omega)}
\leq C \ell^{-1} \|r\|_{L^2(\Omega)}\label{est1}
\ee
with another constant $C=C(\delta, \tau_1,\lambda)$ independent of $\ell>M$ (since $k_1$ is from a bounded region
in the complex plane).
It suffices to estimate the integral and the sum in \eqref{formA} separately. Using \eqref{est1}, the
$L^2$-norm of the integral over $[-\pi, \pi]+i\tau_1$ is easily estimated by $C \ell^{-1} \|r\|_{L^2(\Omega)}$
with another
constant $C=C(\delta, \tau_1,\lambda)$ independent of $\ell>M$. On the
other hand, each pole $q_j^+(k_2)$ lies in exactly one of the contours $\Gamma^\pm_{m_2}$ and hence the
residue $\res(H(\cdot, k_2) r, q_j^+(k_2))$ can be expressed as
$$
\frac{1}{2\pi i} \oint_{\Gamma^\pm_{m_2}} H(k, k_2)r~dk
$$
which together with \eqref{est1}, immediately implies the estimate
$$
\left\|2\pi i \sum_{j=1}^{N^+} \res(H(\cdot, k_2) r, q_j^+(k_2))\right\|_{L^2(\Omega)}\leq C \ell^{-1} \|r\|_{L^2(\Omega)}
$$
with another constant $C=C(\delta, \tau_1,\lambda)$ independent of $\ell>M$.
\end{proof}

\section{Proof of the main theorem}\label{sec_proofmain}
\begin{proof}[Proof of theorem \ref{mainth}]
Consider the analytic Fredholm equation \eqref{eqfinalfred}
for the unknown function $v\in L^2(\Omega)$, and with $k_2\in \Ne(\Gamma)$.
By theorem \ref{large}, \eqref{eqfinalfred} has only the trivial solution if
$k_2\in \Ne(\Gamma)$ with $\Im k_2 = \frac{\pi}{2}+\ell$ and
$\ell\in 2\pi \N$ is sufficiently large. By the analytic Fredholm theorem, the set where \eqref{eqfinalfred}
has a nontrivial solution is discrete in $\Ne(\Gamma)$.

Now recall that \eqref{eq6} has a nontrivial solution $u\in H^2(\R^2)$, and
thus by the arguments of section \ref{sec_reform} and propositions \ref{proponw}
and \ref{propeqonOmega}, \eqref{eqfinalfred} has a nontrivial solution for all
$k_2\in \cP$. But $\cP\cap Z_0\subset \Ne(\Gamma)$ has positive one-dimensional
Lebesgue measure, yielding a contradition.
\end{proof}

\section*{Acknowledgements}
We would like to thank Dirk Hundertmark for valuable comments on the manuscript. 
Part of this work was done while the authors were staying at the Isaac Newton Institute
(Cambridge, UK). The authors would like to thank the Isaac Newton Institute for its
hospitality. Finally, we would like thank the anonymous reviewer for his very careful reading
of the manuscript and his many constructive suggestions, in particular those to simplify the  
proof of theorem \ref{freeofpoles}.

\appendix
\section{Bloch variety}

In the following, we recall a basic fact about the periodic operator $-\frac{1}{\eps_0(\bx)}\Delta$.
The Bloch variety $B$ for $-\frac{1}{\eps_0(\bx)}\Delta$ is defined by
\beq\label{BlochVar}\nn
B := \{ (\bk, \lambda) \in \C^2\times \C : -\frac{1}{\eps_0(\bx)}\Delta_{\bk} u=\lambda
u~~\text{has a nontrivial solution}~u\in H_{\text{per}}^2(\Omega)\}.
\eeq
Since $\Omega$ is bounded, $(-\frac{1}{\eps_0}\Delta_\bk - \lambda)^{-1}$ exists if and only if
the equation $(-\frac{1}{\eps_0}\Delta_\bk - \lambda)u = 0$ has no nontrivial solution $u\in H^2_{\text{per}}(\Omega)$.
In other words, $T(\bk)$ exists if and only if $(\bk, \lambda)\in B$.

A theorem by P. Kuchment (see \cite{Kuch1}, theorem 4.4.2) states that $B$ is the zero set of
an analytic function $F$:
\begin{theorem}\label{thmF}
There exists an analytic function $F: \C^2\times \C\to \C, (\bk, \lambda) \mapsto F(\bk, \lambda)$
which is $2\pi\Z^2$-periodic in $\bk$ such that $B$ is the zero set of $F$. As a consequence,
$k_1$ is a pole of $T(\cdot, k_2)$ if and only if $F(k_1, k_2, \lambda) = 0$.
\end{theorem}

\section{Proof of theorem \ref{thmpolesanalytic}}
Let $k_2^0\in Z$. Then by theorem \ref{thmF}, the poles of $T(\cdot, k_2)$ are
given by the zeroes of an analytic function $F(\cdot, k_2, \lambda)$ for
$k_2\in \C$.  Fix any $k_1^0\in D$ such that $F(k_1^0, k_2^0, \lambda)=0$, i.e. $k_1^0$ is a pole.
The function $k_1\mapsto F(k_1, k_2^0, \lambda)$ is not identically
zero, since there are no poles on the line segments
$[-\pi, \pi]+i \tau_1$. Hence $\frac{\partial^m F}{\partial k_1^m}(k_1^0, k_2^0)\neq 0$ for some $m\in \N_0$
and by the Weierstrass preparation theorem (see e.g. \cite{Hoer}) we may write
\beq\label{weier}
F(k_1, k_2) = a(k_1, k_2)[(k_1-k_1^0)^{m} + \sum_{j=0}^{m-1} b_j(k_2) (k_1-k_1^0)^j]
\eeq
with analytic functions $b_j$ and $a$, $a(k_1^0, k_2^0)\neq 0$. The representation
\eqref{weier} holds for $k_1, k_2$ close to $k_1^0, k_2^0$, say for $|k_1-k_1^0| < \varepsilon$.
By well-known facts from complex analysis (\cite{Knopp, Goursat, Kato}) the number of
zeros of a polynomial with analytic coefficients depending
on $k_2$ is constant, except when $k_2$ is in a discrete set of exceptional points. Moreover,
the zeroes are algebroidal functions of $k_2$ given locally by Puiseux series
(i.e. can be written as power series in $(k_2-k_2^0)^{1/p}$ with some $p \geq 1$).
On simply connected domains away from the exceptional set, the zeroes can
be represented by analytic functions.

By choosing a sufficiently small neighborhood $\Ne_1(k_2^0)$ we may ensure that
$|F(k_1, k_2)| > 0$ for all $k_2\in \Ne_1(k_2^0), |k_1-k_1^0| = \varepsilon$. This implies that
under small variation in $k_2$, no zeroes of $F$ can cross the the circle $\{|k_1-k_1^0|=\varepsilon\}$.
This means that the number of zeroes inside this small circle is constant, except when $k_2$ is
an exceptional point.

Now apply the foregoing to each of the \emph{finitely} many poles in $D$. We obtain a system of small
disjoint circles in $D$ (around each of the zeroes of $F(\cdot, k_2^0)$) and we see that for all $k_2^0$ there
exists a neighborhood $\Ne$ of $k_2^0$ and a discrete set $\Ex$ inside $\Ne$ such that the number of zeroes
of $F(\cdot, k_2)$ contained in the circles is constant, except when $k_2\in \Ex$. Since the line segments
$[-\pi,\pi]\pm i\tau_1$ always stay free of zeroes of $F$, the number of zeroes in $D$ is equal to the
number of zeroes inside the system of small circles. This proves that the number of poles inside $D$ is
\emph{locally} constant, except on $\Ex$. An easy argument involving overlapping discs yields statement
of theorem \ref{thmpolesanalytic} on the existence of $N$.

The statement on the existence of the analytic functions $\{p_j\}$ is proved similarly, by working locally
and recalling the properties of zeroes of polynomials with analytic coefficients mentioned above.

\section{Estimates for the symbol of $-\Delta_{\bk}$ for complex $\bk$}
Recall that $s(\boldm, \bk) = (\boldm+\bk)^2$ is the symbol of the
operator $-\Delta_{\bk}$ in the Fourier series representation. The
following estimates for the symbol, are completely elementary, yet
they play a crucial role in our whole argument.

\begin{theorem}\label{prop1} For $\bxi=(\xi_1,\xi_2),~\bet=(\eta_1,\eta_2)\in\R^2$,
$\boldm=(m_1,m_2)\in 2\pi\Z^2$ the following estimates hold: \be
\label{hammer1} |s(\boldm, \bxi +i \bet)|^2 \geq
[(m_2+\xi_2)^2-\eta_1^2]^2+[(m_1+\xi_1)^2-\eta_2^2]^2 \ee \be
\label{hammer3} |s(\boldm, \bxi +i \bet)|^2 \geq 2[(m_1 +
\xi_1)\eta_1 + (m_2+\xi_2)\eta_2]^2 \ee
\end{theorem}
The proof is an elementary calculation using the identity
\bea
&&(\chi_2^2 - \eta_1^2 + \chi_1^2 - \eta_2^2)^2 + 4(\chi_1 \eta_1 + \chi_2 \eta_2)^2 \\
&=&(\chi_2^2 - \eta_1^2)^2 + (\chi_1^2 - \eta_2^2)^2 +
2(\chi_1\eta_1 + \chi_2\eta_2)^2+2(\chi_2\chi_1+\eta_1\eta_2)^2.
\eea
for $\chi_i = m_i + \xi_i,~~i = 1,2$.

For $\xi_2\in[\pi-\delta,\pi+\delta]$  define $\mathcal{J}_+,\mathcal{J}_-:2\pi\Z^2 \to \C$
by
\bea
\mathcal{J}_\pm (m_1,m_2)&=&\pm(\frac{\pi}{2}+\ell)-m_1\mp i(m_2+\xi_2).
\eea
Note that $\mathcal{J}_+$ and $\mathcal{J}_-$ are one to one and
$$\mathcal{J}_+(2\pi\Z^2)\cap \mathcal{J}_-(2\pi\Z^2)=\emptyset.$$
In the next lemma, the poles of $(-\Delta_{(\cdot, \xi_2+i \eta_2)})^{-1}$ are determined.
The proof is a simple computation using the above defined $\mathcal{J}_\pm$.

\begin{lemma}\label{lemmanullstellen}
Let $k_2=\xi_2+i\eta_2$ with $\xi_2\in[\pi-\delta,\pi+\delta]$ and
$\eta_2=\frac{\pi}{2}+\ell$, $\ell\in 2\pi\N_0$ be fixed. Then
\begin{enumerate}
\item [(i)] $s(\boldm,(k_1,\xi_2+i\eta_2))=0$ if and only if $k_1=\mathcal{J}_+(\boldm)$ or $k_1=\mathcal{J}_-(\boldm)$.
\item [(ii)] $\ker (-\Delta_{(k_1,\xi_2+i\eta_2)})=\operatorname{span}\,\{e^{i\boldm\cdot\bx}\},$
where $\boldm\in 2\pi\Z^2$ is uniquely determined by the condition
$k_1=\mathcal{J}_+(\boldm)$ or $k_1=\mathcal{J}_-(\boldm)$. If there
is no $\boldm\in 2\pi\Z^2$ satisfying $k_1=\mathcal{J}_+(\boldm)$ or
$k_1=\mathcal{J}_-(\boldm)$, then $-\Delta_{(k_1,\xi_2+i\eta_2)}$ is
invertible.
\end{enumerate}
As a consequence, for any pole $k_1$ of
$-\Delta_{(\cdot,\xi_2+i\eta_2)}^{-1}$ we have either
$k_1=\mathcal{J}_+(\boldm)$ or $k_1=\mathcal{J}_-(\boldm)$ with a
uniquely determined $\boldm\in 2\pi\Z^2$. Moreover, if $k_1$ is a
pole of $-\Delta_{(\cdot,\xi_2+i\eta_2)}^{-1}$, then $k_1+m$, $m\in
2\pi\Z$ is also a pole.
\end{lemma}

\begin{lemma}\label{lemsymbolest} There exists a $C=C(\delta,\tau_1)>0$ and a $M=M(\delta,\tau_1)>0$ such that for all
$\ell\in 2\pi\N$, $\ell>M$, all $k_1\in\Lines$, and all
$\xi_2\in[\pi-\delta,\pi+\delta]$ the following estimate for the
symbol of $-\Delta_{\bk}$ holds: \beq\label{eqsymbolest1}\nn
\left|s(\boldm,(k_1,\xi_2+i\left(\frac{\pi}{2}+\ell\right)))\right|\ge
C\ell\quad (\boldm\in 2\pi \Z^2). \eeq As a consequence,
\beq\label{eqlaplacenorm} \no{(-\Delta_{(k_1,
\xi_2+i\left(\frac{\pi}{2}+\ell\right))})^{-1}}\le C \ell^{-1} \eeq
for all $k_1\in\Lines$, $\ell> M$,
$\xi_2\in[\pi-\delta,\pi+\delta]$.
\end{lemma}
\begin{proof}[Proof of lemma \ref{lemsymbolest}] In total we have to consider four cases:
\begin{enumerate}
\item [1.] Vertical lines: $k_1=(\pm\frac{\pi}{2}\pm 2\delta)+i\nu$ $(\nu\in\R)$
\begin{enumerate}
\item[Case 1.1] $\boldm=(\pm\ell,m_2)$
\item[Case 1.2] $\boldm=(m_1,m_2)$ with $m_1\neq\pm\ell$
\end{enumerate}
\item [2.] Horizontal lines: $k_1=\mu+i\nu$ $(\mu\in[-\pi,\pi],~\nu\in 2\pi\Z,~|\nu|\le\tau_1)$
\begin{enumerate}
\item[Case 2.1] $\boldm=(\pm\ell,m_2)$
\item[Case 2.2] $\boldm=(m_1,m_2)$ with $m_1\neq\pm\ell$
\end{enumerate}
\end{enumerate}
For the case 1.1, we use the estimate \eqref{hammer1} to obtain \bea
\left|s(\boldm,(k_1,\xi_2+i\left(\frac{\pi}{2}+\ell\right)))\right|^2&\ge&
[(\pm\ell+(\pm\frac{\pi}{2}\pm 2\delta))^2-\left(\frac{\pi}{2}+\ell\right)^2]^2\\
&=&[\pm 2\ell(\pm\frac{\pi}{2}\mp\frac{\pi}{2}\pm2\delta)+(\pm\frac{\pi}{2}\pm 2\delta)^2-\frac{\pi^2}{4}]^2\\
&\ge&C(\delta)\ell^2
\eea
for sufficiently large $\ell\in 2\pi\N$, since $(\pm\frac{\pi}{2}\mp\frac{\pi}{2}\pm2\delta)\neq 0$ by the choice
$0<2\delta<\frac{\pi}{2}$.\\
To treat the cases 1.2 and 2.2 we consider the intervals
$$I_{m}:=(m+[-\pi,\pi])^2=\{(m+\eta)^2:~\eta\in[-\pi,\pi]\},$$
where $m\in 2\pi\Z$. Then $I_{m}=I_{-m}$ and
$\max I_{|m|}=\min I_{|m|+2\pi}$. So the intervals $I_{|m|}$ and $I_{|m|+2\pi}$ are adjacent
and the union of all $I_{m}$ is $[0,\infty)$. Since $\left(\frac{\pi}{2}+\ell\right)^2\in I_\ell$
and $m_1\neq \pm \ell$ we have for sufficiently large $\ell$
\ben
\nonumber\text{dist}(\left(\frac{\pi}{2}+\ell\right)^2,I_{m_1})
&\ge&\min\left\{\left(\frac{\pi}{2}+\ell\right)^2-(\ell-\pi)^2,(\ell+\pi)^2-\left(\frac{\pi}{2}+\ell\right)^2
\right\}\\
\label{eqdist}&\ge&(\pi\ell+\frac{3}{4}\pi^2)\ge \pi \ell. \een
 Using \eqref{hammer1} we obtain \bea
\left|s(\boldm,(\mu+i\nu,\xi_2+i\left(\frac{\pi}{2}+\ell\right)))\right|^2
&\ge&[(m_1+\mu)^2-\left(\frac{\pi}{2}+\ell\right)^2]^2\ge C^2\ell^2
\eea by \eqref{eqdist} since $(m_1+\mu)^2\in I_{m_1}$. This proves
the desired estimate in the case 2.2. In the case 1.2, the proof is
the same since again by estimate \eqref{hammer1} \bea
\left|s(\boldm,((\pm\frac{\pi}{2}\pm
2\delta)+i\nu,\xi_2+i\left(\frac{\pi}{2}+\ell\right)))\right|^2
&\ge&[(m_1+(\pm\frac{\pi}{2}\pm
2\delta))^2-\left(\frac{\pi}{2}+\ell\right)^2]^2 \eea
and $(m_1+(\pm\frac{\pi}{2}\pm 2\delta))^2\in I_{m_1}$.\\
For the case 2.1 we use the estimate \eqref{hammer3}. \bea
\left|s(\boldm,(\mu+i\nu,\xi_2+i\left(\frac{\pi}{2}+\ell\right)))\right|^2
&\ge&2[(\pm\ell+\mu)\nu+(m_2+\xi_2)\left(\frac{\pi}{2}+\ell\right)]^2\\
&=&2\left(\frac{\pi}{2}+\ell\right)^2\left[m_2+\frac{\pm\ell+\mu}{\left(\frac{\pi}{2}+\ell\right)}\nu+\xi_2\right]^2.
\eea
$m_2+\frac{\pm\ell+\mu}{\left(\frac{\pi}{2}+\ell\right)}\nu$ converges to $m_2\pm \nu\in 2\pi\Z$ as
$\ell\to\infty$ (uniformly with respect to
$\mu\in[-\pi,\pi]$ and $\nu\in 2\pi\Z,~|\nu|\le\tau_1$). Thus $m_2+\frac{\pm\ell+\mu}{\left(\frac{\pi}{2}+\ell\right)}\nu$ is contained in a
sufficiently small neighborhood of the grid $2\pi\Z$ for sufficiently large $\ell$.
Since $\xi_2\in[\pi-\delta,\pi+\delta]$ with $0<\delta<\frac{\pi}{4}$, there exists a constant
$C(\delta,\tau_1)>0$ (independent of $m_2$) such that for sufficiently large $\ell$
$$\left[m_2+\frac{\pm\ell+\mu}{\left(\frac{\pi}{2}+\ell\right)}\nu+\xi_2\right]^2\ge C(\delta,\tau_1)>0.$$ Then, for
sufficiently large $\ell$ \bea
\left|s(\boldm,(\mu+i\nu,\xi_2+i\left(\frac{\pi}{2}+\ell\right)))\right|^2\ge
C(\delta,\tau_1)\ell^2 \eea holds, with another constant
$C(\delta,\tau_1)>0$. \end{proof}

Using the same Neumann series argument as in the proof of theorem \ref{freeofpoles} we obtain the following

\begin{corollary}\label{corTbound}
There exists a $C=C(\delta,\tau_1,\lambda)>0$ and a $M=M(\delta,\tau_1,\lambda)>0$ such that for all
$\ell\in 2\pi\N$, $\ell>M$,
all $k_1\in\Lines$, and all $\xi_2\in[\pi-\delta,\pi+\delta]$ the following estimate holds:
\beq\nn
\no{T(k_1,\xi_2+i\left(\frac{\pi}{2}+\ell\right))}\le C\ell^{-1}.
\eeq
\end{corollary}

\section{Proof of proposition \ref{propmero}}
\begin{proof}
Fix $k_2$. First we work locally in $k_1$, i.e. fix a complex ball $B(k_1^{*}, R)$
around $k_1^{*}\in \C$. For $\mu > 0$ large enough, the quadratic form
$$(u, v) \mapsto \la (-\Delta_{(k_1, k_2)}+\mu) u, v\ra$$ is uniformly coercive
for all $k_1\in B(k_1^{*}, R)$ and hence $(-\Delta_{(k_1, k_2)}+\mu)^{-1}$
exists as a bounded operator between $L^2(\Omega)$ and $H^2_{\text{per}}(\Omega)$
(this is easily seen using Fourier series on the cube $\Omega$).
$(-\Delta_{(k_1, k_2)}-\lambda \eps_0)^{-1}$ exists as a bounded operator if and only if
the operator
\beq\label{Fredholm}
(I + (-\Delta_{(k_1, k_2)}+\mu)^{-1} (\eps_0 - \mu)  )^{-1}
\eeq
exists as a bounded operator from $L^2(\Omega)$ onto $L^2(\Omega)$. Note that
$(-\Delta_{(k_1, k_2)}+\mu)^{-1} (\eps_0 - \mu):L^2(\Omega)\to L^2(\Omega)$ is compact.
So by the analytic Fredholm theorem (see \cite{Reed}), \eqref{Fredholm} exists either
nowhere on $B(k_1^{*}, R)$ or on $B(k_1^*, R) \setminus \Sigma$, where $\Sigma$
is a non-accumulating discrete set. As a consequence, the same holds true for
$(-\Delta_{(k_1, k_2)}-\lambda \eps_0)^{-1}$ on $B(k_1^{*}, R)$. A covering argument
using overlapping balls shows that the analytic Fredholm alternative holds for $(-\Delta_{(\cdot, k_2)}-\lambda \eps_0)^{-1}$
on the whole of $\C$. This proves the claim, since $T(k_1, k_2)$ is assumed to exist for
at least one $k_1\in \C$.
\end{proof}

\end{document}